\documentclass{article}
\usepackage{amsfonts, amsmath, amssymb, amscd, amsthm, graphicx,enumerate}
\usepackage{hyperref,xypic}
\usepackage[usenames]{color}
\usepackage{tikz}
\usepackage{thmtools,thm-restate}
\usepackage{subcaption}

\makeatletter
\def\blfootnote{\xdef\@thefnmark{}\@footnotetext}
\makeatother

\newtheorem{thm}{Theorem}[section]
\newtheorem{cor}[thm]{Corollary}
\newtheorem{lem}[thm]{Lemma}

\newtheorem{claim}[thm]{Claim}
\theoremstyle{definition}
\newtheorem{defn}[thm]{Definition}
\theoremstyle{remark}
\newtheorem{rem}[thm]{Remark}

\newfont{\eufm}{eufm10}

\renewcommand{\phi}{\varphi}

\newcommand{\N}{\mathbb N}
\newcommand{\Z}{\mathbb Z}
\newcommand{\R}{\mathbb R}
\newcommand{\hyp}{\mathbb H}
\newcommand{\llangle }{\langle\hspace{-.7mm}\langle }
\newcommand{\rrangle }{\rangle\hspace{-.7mm}\rangle }

\newcommand{\Mod}{\operatorname{Mod}}
\newcommand{\PMod}{\operatorname{PMod}}
\newcommand{\Fix}{\operatorname{Fix}}
\newcommand{\SL}{\operatorname{SL}}
\newcommand{\Aut}{\operatorname{Aut}}
\newcommand{\diam}{\operatorname{diam}}
\newcommand{\C}{\mathcal{C}}

\newcommand{\nest}{\sqsubseteq}

\renewcommand{\H}{\mathcal{H}}

\def\mc {\mathcal}
\def\onto {\twoheadrightarrow}

\begin{document}

\title{Largest hyperbolic actions and quasi-parabolic actions in groups}
\author{Carolyn R. Abbott and Alexander J. Rasmussen}
\date{}
\maketitle

\begin{abstract}
The set of equivalence classes of cobounded actions of a group on different hyperbolic metric spaces carries a natural partial order. The resulting poset thus gives rise to a notion of the ``best'' hyperbolic action of a group as the largest element of this poset, if such an element exists. We call such an action a \textit{largest hyperbolic action}. While hyperbolic groups  admit largest hyperbolic actions, we give evidence in this paper that this phenomenon is rare for non-hyperbolic groups. In particular, we prove that many families of groups of geometric origin do not have largest hyperbolic actions, including for instance many 3-manifold groups and most mapping class groups. Our proofs use the quasi-trees of metric spaces of Bestvina--Bromberg--Fujiwara, among other tools. In addition, we give a complete characterization of the poset of hyperbolic actions of Anosov mapping torus groups, and we show that mapping class groups of closed surfaces of genus at least two have hyperbolic actions which are comparable only to the trivial action.
\end{abstract}

\section{Introduction}

A fruitful approach for proving algebraic, geometric, and algorithmic facts about groups is to study their (isometric) actions on metric spaces which exhibit large-scale negative curvature --- so-called Gromov hyperbolic metric spaces. Among many other things, such actions may be used to study quotients of groups (\cite{dgo,exotic}), bounded cohomology of groups (\cite{bchmcg}), and isoperimetric functions of their Cayley graphs.

Owing to the importance of actions on hyperbolic metric spaces, it is natural to try to find a ``best'' action of a given group on a hyperbolic metric space. We will explain what we mean by this precisely below, but for now one may think of a hyperbolic action $G\curvearrowright X$ as ``best'' when any hyperbolic action $G\curvearrowright Y$ may be obtained from $G\curvearrowright X$ be applying some simple collapsing operations.

In fact, this goal is slightly too broad, as any countable group admits many actions on hyperbolic metric spaces with a global fixed point on the boundary (the \textit{parabolic} actions) which are somewhat trivial and impossible to classify. Hence we restrict our attention to cobounded actions. Given a group $G$, the equivalence classes of cobounded actions of $G$ on hyperbolic spaces form a poset $\mathcal{H}(G)$ (see Section \ref{sec:hypstructures} for the precise definition). By a best hyperbolic action, we mean the largest element of the poset $\mathcal{H}(G)$, if it exists  (an element of a poset is \textit{largest} if it is comparable to and greater than any other element of the poset).  When the group $G$ is hyperbolic, the poset $\mathcal{H}(G)$ always contains a largest element, which corresponds to the action of $G$ on its Cayley graph with respect to a finite generating set. In other words, if $G$ is hyperbolic then every cobounded hyperbolic action of $G$ may be obtained (up to equivalence) by equivariantly collapsing subspaces of its Cayley graph.

The purpose of this paper is to provide evidence that the existence of a largest element in $\mc H(G)$ is rare when $G$ is not hyperbolic.  %the phenomenon of $\mathcal{H}(G)$ containing a largest element is rare when $G$ is not hyperbolic. 
%We prove:

\begin{thm}
\label{mainthm}
The poset $\mathcal{H}(G)$ doesn't contain a largest element when $G$ is any one of the following groups:
\begin{itemize}
\item the mapping class group of an orientable finite-type surface $S$ which is not a sphere minus $\leq 4$ points or a torus minus $\leq 1$ point,
\item a non-free right-angled Artin group,
\item the fundamental group of a flip graph manifold with at least two pieces in its JSJ decomposition,
\item the fundamental group of a finite-volume cusped hyperbolic 3--manifold,
\item the fundamental group of the mapping torus of an Anosov homeomorphism of the torus,
\item a Baumslag-Solitar group,
\item a finitely generated solvable group with abelianization of rank $>1$.
\end{itemize}
\end{thm}

We also prove a further structural theorem in the case of mapping class groups.

\begin{restatable}{thm}{maxlineal}
\label{thm:maxlineal}
Let $S$ be an orientable closed surface of genus $\geq 2$. Then $\mathcal{H}(\Mod(S))$ contains elements which are comparable only to the equivalence class of the trivial action on a point.
\end{restatable}

In the case of Anosov mapping torus groups, we give a complete characterization of the poset $\H(G)$.

\begin{thm}
\label{anosovposet}
Let $G$ be the fundamental group of the mapping torus of an Anosov homeomorphism of the torus. Then $\mathcal{H}(G)$ consists of two incomparable quasi-parabolic structures, which dominate a single lineal structure, which in turn dominates a single elliptic structure. The quasi-parabolic structures correspond to actions of $G$ on the hyperbolic plane, $\hyp^2$. See Figure \ref{anosovdiagram}.
\end{thm}
 
\begin{figure}[h]
\centering
 
\begin{tikzpicture}[scale=0.35]

\node[circle, draw, minimum size=0.8cm] (hypleft) at (-3,3) {$ \hyp^2$};
\node[circle, draw, minimum size=0.8cm] (hypright) at (3,3) {$ \hyp^2$};
\node[circle, draw, minimum size=0.8cm] (lin) at (0,0) {$ \R$};
\node[circle, draw, minimum size=0.8cm] (triv) at (0,-3) {$ *$};

\draw[thick] (triv) -- (lin) -- (hypleft);
\draw[thick] (lin) -- (hypright);
\end{tikzpicture}

\caption{The poset $\mathcal{H}(G)$ when $G$ is an Anosov mapping torus.}
\label{anosovdiagram}.
\end{figure}
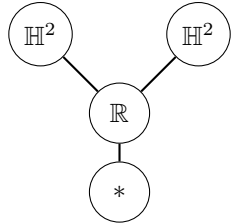

% \subsection{Hyperbolic structures on groups}
% 
%We will say that an action of $G$ on a metric space $X$ by isometries is \textit{cobounded} if for some (equivalently, any) $x\in X$ there exists $R>0$ such that the $R$-neighborhood of the orbit $Gx$ is all of $X$. We will say that two actions $G\curvearrowright X$ and $G\curvearrowright Y$ are \textit{equivalent} if there exists a coarsely $G$-equivariant quasi-isometry $X\to Y$. Thus we may consider the set $\mathcal{H}(G)$ of \textit{equivalence classes of cobounded hyperbolic actions of $G$}.
% 
%The set $\mathcal{H}(G)$ carries a partial order defined as follows. If $G\curvearrowright X$ and $G\curvearrowright Y$ are cobounded actions of $G$ on hyperbolic metric spaces, then we say that $G\curvearrowright X$ \textit{dominates} $G\curvearrowright Y$ and write $G\curvearrowright X\succeq G\curvearrowright Y$ if there exists a coarsely $G$-equivariant coarsely Lipschitz map $X\to Y$. This partial order on actions descends to a partial order on equivalence classes. Thus, the partial order captures the informal relation of ``collapsing equivariant families of subspaces.'' It is not hard to check that if $G$ is hyperbolic then the equivalence class of the action of $G$ on its Cayley graph with respect to any finite generating set is largest.

\subsection{About the proofs}

Denote by $\preceq$ the partial order on hyperbolic actions of a group. We prove the following simple lemma.

\begin{restatable}{lem}{mainlem}
\label{lem:mainlem}
Let $G$ be a group. Let $a,b\in G$ be elements which commute and let $G\curvearrowright X$, $G\curvearrowright Y$ be two actions on hyperbolic spaces such that
\begin{itemize}
\item $a$ acts loxodromically and $b$ acts elliptically in the action $G\curvearrowright X$,
\item $b$ acts loxodromically in the action $G\curvearrowright Y$.
\end{itemize} 
Then there does not exist an action $G\curvearrowright Z$ with $Z$ hyperbolic such that $G\curvearrowright X \preceq G\curvearrowright Z$ and $G\curvearrowright Y\preceq G\curvearrowright Z$.
\end{restatable}
 
The lemma applies to give the proof for most of the groups $G$ mentioned Theorem \ref{mainthm}. A notable exception is when $G$ is the mapping torus of an Anosov map of the torus. Although the proof of Theorem \ref{mainthm} reduces to Lemma \ref{lem:mainlem} in most cases, the methods of proof in each case are quite different. Moreover, the difficulty of each case varies immensely. For right-angled Artin groups, Baumslag-Solitar groups, and solvable groups, the proofs are algebraic and relatively straightforward. For mapping class groups and fundamental groups of flip graph manifolds, the proofs are much more complicated and involve the \textit{quasi-trees of metric spaces} of Bestvina--Bromberg--Fujiwara (\cite{bbf}). In the case of mapping class groups, much of the relevant work was done in \cite{scl}, while in the case of flip graph manifolds we build up the relevant quasi-trees and actions mostly from scratch. The application of the Bestvina--Bromberg--Fujiwara machinery in this case may be of independent interest.

\subsection{Organization}

In Section \ref{sec:background}, we give necessary background on hyperbolic structures on groups, quasimorphisms, and quasi-trees of metric spaces. After this, the remaining sections of the paper may be read independently of each other.

We introduce confining subsets and their connections with quasi-parabolic structures in Section \ref{sec:confsets}. This material is used only in Section \ref{sec:anosov}. We introduce the quasi-tree of metric spaces machinery of Bestvina--Bromberg--Fujiwara in Section \ref{sec:bbf}. This material is used only in Sections \ref{sec:mcgs} and \ref{sec:flipgraph}.

In Section \ref{sec:mainlem}, we prove Lemma \ref{lem:mainlem} and give the proof of Theorem \ref{mainthm} for right-angled Artin groups, Baumslag-Solitar groups, and solvable groups.  These are the cases in which the application of Lemma \ref{lem:mainlem} is most straightforward; the proofs are all algebraic.

%In Sections \ref{sec:raags} -- \ref{sec:solvable} we give the proof of Theorem \ref{mainthm} for right-angled Artin groups, Baumslag-Solitar groups, and solvable groups. In these sections, the Main Lemma \ref{lem:mainlem} applies most easily and the proofs are all algebraic.
In Section \ref{sec:mcgs}, we  apply  the Bestvina--Bromberg--Fujiwara machinery and prove all of our results on mapping class groups.
In Section \ref{sec:flipgraph}, we apply the Bestvina--Bromberg--Fujiwara machinery to prove Theorem \ref{mainthm} for flip graph manifold groups.
In Section \ref{sec:hyp3man}, we use Dehn filling to prove Theorem \ref{mainthm} for cusped hyperbolic 3--manifold groups.
In Section \ref{sec:anosov} we completely describe $\H(G)$ when $G$ is an Anosov mapping torus group, and thus prove Theorem \ref{mainthm} in this case as well.

\textbf{Acknowledgements.} We would like to thank Jason Behrstock and Yair Minsky for helpful conversations related to this work. The first author was partially supported by NSF Awards DMS-1803368 and DMS-2106906. The second author was partially supported by NSF Award DMS-1610827.

%%%%%%%%%%%%%%%%%%%%%%%%%%%%%%%%%%%%%%%%%%%%%%%%%%%%%%%%%%%%%%%%%%%%%%%%%%%%%%%%%%%%%%%%%%%%%%%%%%%%%%%%%%%%%%%%%%%%%%%%%%%%%%%%%%%%%%%%%%%%%%%%%%%%%%%
\section{Background}\label{sec:background}

\subsection{Actions on hyperbolic spaces} \label{sec:actions}

Given a metric space $X$, we denote by $d_X$ the distance function on $X$. A map $f\colon X\to Y$ between metric spaces $X$ and $Y$ is a \emph{quasi-isometric embedding} if there is a constant $C>0$ such that for all $x,y\in X$, \[\frac1C d_X(x,y)-C\leq d_Y(f(x),f(y))\leq  Cd_X(x,y)+C.\] If, in addition, $Y$ is contained in the $R$--neighborhood of the image $f(X)$ for some $R>0$, then $f$ is called a \emph{quasi-isometry}. If $f\colon X\to Y$ satisfies only \[d_Y(f(x),f(y))\leq Cd_X(x,y)+C,\] then $f$ is called \textit{$C$--coarsely Lipschitz}. If a group $G$ acts (by isometries) on $X$ and $Y$, then a map $f\colon X\to Y$ is \emph{coarsely $G$--equivariant} if for every $x\in X$ we have \[\sup_{g\in G} d_Y(f(gx),gf(x))<\infty.\]

We will assume that all actions in this paper are by isometries. The action of a group $G$ on a metric space $X$ is \emph{cobounded} if for some (equivalently any) $x\in X$ there exists $R>0$ such that $X=B_R(Gx)$ where $Gx$ denotes the orbit of $x$ under $G$.

Given an action $G\curvearrowright X$ with $X$ hyperbolic, an element $g\in G$ is \emph{elliptic} if it has bounded orbits; \emph{loxodromic} if the map $\Z\to X$ given by $n\mapsto g^nx_0$ for some (equivalently, any) $x_0\in X$ is a quasi-isometric embedding; and \emph{parabolic} otherwise.

Any group action on a hyperbolic space falls into one of finitely many types depending on the number of fixed points on the boundary and the types of isometries defined by various group elements. This classification was described by Gromov in \cite{gro}. The action $G\curvearrowright X$ (where $X$ is hyperbolic) is
\begin{itemize}
\item \textit{elliptic} if $G$ has a bounded orbit in $X$;
\item \textit{lineal} if $G$ fixes two points of $\partial X$;
\item \textit{parabolic} if $G$ fixes a unique point of $\partial X$ and no element of $G$ acts as a loxodromic isometry of $X$;
\item \textit{quasi-parabolic} if $G$ fixes a unique point of $\partial X$ and at least one element of $G$ acts as a loxodromic isometry; and
\item \textit{general type} if $G$ doesn't fix any point of $\partial X$ and at least one element of $G$ acts as a loxodromic isometry.
\end{itemize}

\subsection{Hyperbolic structures} \label{sec:hypstructures}

In this section, we review the construction of the poset of hyperbolic structures of a group from \cite{abo}.  Fix a group $G$. For any (possibly infinite) generating set $S$ of $G$,  let $\Gamma(G,S)$ be the Cayley graph of $G$ with respect to the generating set $S$, and let $\|\cdot\|_S$ denote the word norm on $G$ with respect to $S$.
Given two generating sets $S,T$ of a group $G$, we say $T$ is \emph{dominated by} $S$, written  $T\preceq S$, if \[\sup_{g\in S}\|g\|_T<\infty.\]  It is clear that $\preceq$ is a preorder on the set of generating sets of $G$ and so induces the equivalence relation $S\sim T$ if and only if $T\preceq S$ and $S\preceq T$. Let $[S]$ be the equivalence class of a generating set. Then the preorder $\preceq$ induces a partial order  $\preccurlyeq$ on the set of all equivalence classes of generating sets of $G$ via $[S]\preccurlyeq [T]$ if and only if $S\preceq T$.

%Given a generating set $S$, we denote by $\Gamma(G,S)$ the \textit{Cayley graph} of $G$ with respect to $S$.

\begin{defn}
Given a group $G$, the \emph{poset of hyperbolic structures on $G$} is defined to be \[\mathcal H(G):= \{[S]\mid G=\langle S\rangle \textrm{ and } \Gamma(G,S) \textrm{ is hyperbolic}\},\] equipped with the partial order $\preccurlyeq$.
\end{defn}

Notice that since hyperbolicity is a quasi-isometry invariant of geodesic metric spaces, the above definition is independent of the choice of representative of the equivalence class $[S]$. Every element $[S]\in\mc H(G)$ gives rise to a cobounded action on a hyperbolic space, namely $G\curvearrowright \Gamma(G,S)$. Moreover, given a cobounded action on a hyperbolic space $G\curvearrowright X$, a standard Schwarz--Milnor argument (see \cite[Lemma~3.11]{abo}) provides a (possibly infinite) generating set $S$ of $G$ such that $\Gamma(G,S)$ is equivariantly quasi-isometric to $X$. We say that two actions $G\curvearrowright X$ and $G\curvearrowright Y$ are \emph{equivalent} if there exists a coarsely $G$--equivariant quasi-isometry $X\to Y$. By \cite[Proposition~3.12]{abo}, there is a one-to-one correspondence between equivalence classes $[S]\in\mathcal H(G)$ and equivalence classes of cobounded actions $G\curvearrowright X$ with $X$ hyperbolic. The partial order on cobounded actions is given by $G\curvearrowright X \preceq G\curvearrowright Y$ if there exists a coarsely $G$--equivariant coarsely Lipschitz map $Y\to X$. This descends to a partial order $\preccurlyeq$ on equivalence classes $[G\curvearrowright X]$. Thus, the partial order captures the informal relation of ``collapsing equivariant families of subspaces.'' It is not hard to check that if $G$ is hyperbolic then the equivalence class of the action of $G$ on its Cayley graph with respect to any finite generating set is largest.

We denote the set of equivalence classes of cobounded elliptic, lineal, quasi-parabolic, and general-type actions by $\mc H_e,\mc H_\ell,\mc H_{qp},$ and $\mc H_{gt}$, respectively. Since parabolic actions cannot be cobounded, we have for any group $G$, \[\mc H(G)=\mc H_e(G)\sqcup\mc H_\ell(G)\sqcup\mc H_{qp}(G)\sqcup\mc H_{gt}(G).\]

\noindent A lineal action of a group $G$ on a hyperbolic space $X$ is \emph{orientable} if no element of $G$ permutes the two limit points of $G$ on $\partial X$. We denote the set of equivalence classes of orientable lineal structures on $G$ by $\mc H_\ell^+(G)$.

\subsection{Quasimorphisms}

A map $q\colon G\to \R$ is a \emph{quasimorphism} if there exists a constant $D\geq 0$ such that for all $g,h\in G$, we have $|q(gh)-q(g)-q(h)|\leq D$. We say that $q$ has \emph{defect at most $D$}. If, in addition, the restriction of $q$ to every cyclic subgroup is a homomorphism, then $q$ is called a \emph{homogeneous} quasimorphism. Every quasimorphism $q$ gives rise to a homogeneous quasimorphism $\rho$ defined by $\rho(g)=\lim_{n\to\infty}\frac{q(g^n)}{n}$; we call $\rho$ the \emph{homogenization of $q$}. Every homogeneous quasimorphism is constant on conjugacy classes. If $q$ has defect at most $D$, then it is straightforward to check that $|q(g)-\rho(g)|\leq D$ for all $g\in G$.

Let $G\curvearrowright X$ be an action on a hyperbolic space with a global fixed point $\xi\in\partial X$. For any sequence $\mathbf x=(x_n)$  in $X$ converging to $\xi$ and any fixed basepoint $s\in X$, we define the associated quasimorphism $q_{\bf x}\colon G\to \mathbb R$ as follows. For all $g\in G$, \begin{equation*} \label{eqn:qchar} q_\mathbf x(g)=\limsup_{n\to\infty}(d_X(gs,x_n)-d_X(s,x_n)). \end{equation*}
Its homogenization $\rho_\mathbf x\colon G\to\mathbb R$ is the \emph{Busemann quasimorphism}. It is known that for any two sequences $\mathbf x,\mathbf y$ converging to $\xi$,  we have $\sup_{g\in G}|q_\mathbf x(g)-q_\mathbf y(g)|<\infty$, and thus we may drop the subscript $\mathbf x$ in $\rho_\mathbf x$. If  $\rho$ is a homomorphism, then the action $G\curvearrowright X$ is called \emph{regular}.

In this paper, we will repeatedly make use of one particular construction of a quasimorphism. Given an action of a group $G$ on a hyperbolic metric space $X$ and $g\in G$ which is loxodromic with respect to the action on $X$, there is a quasimorphism $q$ on $G$ associated to $g$ defined by Bestvina--Fujiwara in \cite{bchmcg} which we call a \textit{Brooks quasimorphism}. In general the quasimorphism $q$ may be badly behaved. For this reason, one must usually impose further dynamical restrictions on the isometry $g$.

\begin{defn}
Let $g\in G$ be loxodromic in the action on $X$ with fixed points $\{g^\pm\} \subset \partial X$. We say that $g$ is \textit{WWPD} if the orbit of $(g^+,g^-)$ under $G$ is discrete in the space $\partial X \times \partial X \setminus \Delta$ where $\Delta$ is the diagonal $\Delta=\{(x,x):x\in \partial X\}$. If $g$ is WWPD then we say that it is \textit{WWPD${}^+$} if, given $h\in G$ such that $h$ fixes the fixed points $g^\pm$ of $g$ as a \textit{set}, we also have $hg^+=g^+$ and $hg^-=g^-$.
\end{defn}

See \cite{wwpd} for this specific formulation of the WWPD property (there called WWPD (2)) and for several other equivalent definitions of WWPD. The WWPD${}^+$ property is strong enough to define a \textit{well-behaved} Brooks quasimorphism (see \cite[Corollary~3.2]{scl}):

\begin{restatable}{prop}{qmconstruction}
\label{prop:qmconstruction}  Let $G\curvearrowright X$ be an action of $G$ on a hyperbolic metric space $X$ with a WWPD${}^+$ element $g$. Then there is a homogeneous quasimorphism $q\colon G\to\R$ such that the following hold:
\begin{enumerate}
\item $q(g)\neq 0$; and
\item $q(h)=0$ for any element $h\in G$ which acts elliptically on $X$.
\end{enumerate}
\end{restatable}

Although the language used in \cite[Corollary~3.2]{scl} is slightly different, Proposition \ref{prop:qmconstruction} is an immediate corollary (and, in fact, \cite[Corollary~3.2]{scl} is stronger than what we stated).

We will use the following lemma several times in this paper:

\begin{lem}[{\cite[Lemma~4.15]{abo}}]
\label{lem:quasiline}
Let $q:G\to \R$ be a nonzero homogeneous quasimorphism. Then there is an action of $G$ on a quasi-line $X$ with the property that $g$ acts loxodromically on $X$ if and only if $q(g)\neq 0$.
\end{lem}

\subsection{Confining subsets}
\label{sec:confsets}

Consider a group $G=H\rtimes_{\alpha} \Z$ where $\alpha\in \Aut(H)$ acts by $\alpha(h)=tht^{-1}$ for any $h\in H$, where $t$ is a generator of $\Z$. Let $Q$ be a \textit{symmetric} subset of $H$.  The following definition is from \cite[Section~4]{ccmt}.

\begin{defn} \label{def:confining}
The action of $\alpha$ is \emph{(strictly) confining $H$ into $Q$} if it satisfies the following three conditions.
\begin{enumerate}[(a)]
\item $\alpha(Q)$ is (strictly) contained in $Q$;
\item $H=\bigcup_{k\geq 0} \alpha^{-k}(Q)$; and
\item $\alpha^{k_0}(Q\cdot Q)\subseteq Q$ for some $k_0\in \Z_{\geq 0}$.
\end{enumerate}
\end{defn}

\begin{rem}
The definition of confining subset given in \cite{ccmt} does not require symmetry of the subset $Q\subset H$. However, according to \cite[Theorem~4.1]{ccmt}, to classify regular quasi-parabolic structures on a group of the above form, it suffices to consider only confining subsets which are symmetric. \end{rem}

\begin{rem} By the discussion after the statement of \cite[Theorem~4.1]{ccmt}, if there is a subset $Q\subseteq H$ such that the action of $\alpha$ is confining $H$ into $Q$ but not strictly confining, then $[Q\cup\{t^{\pm 1}\}]\in \mathcal H_\ell^+(G)$. If the action is strictly confining, then $[Q\cup\{t^{\pm 1}\}]\in \mathcal H_{qp}(G)$.
\end{rem}

In this paper, we will focus primarily on describing subsets $Q$ of $H$ into which the action of $\alpha$ is (strictly) confining $H$. For brevity, we will refer to such $Q$ as \emph{(strictly) confining under the action of $\alpha$.}

\subsection{Quasi-trees of metric spaces}
\label{sec:bbf}

In this section, we review the construction of a projection complex and a quasi-tree of metric spaces from \cite{bbf}.  We begin by giving a canonical example to keep in mind.  Let $G=\pi_1(\Sigma)$ where $\Sigma$ is a closed, hyperbolic surface.  Fix  a simple closed geodesic $\gamma$ on $\Sigma$, and let $\bf Y$ be the set of lifts of $\gamma$ to the universal cover $\mathbb H^2$.  Then for any $Y,Z\in\bf Y$, the nearest-point projection $\pi_Y(Z)$ of $Y$ to $Z$ is uniformly bounded.  Moreover, if $X,Y,Z\in\bf Y$ and the projections of $\pi_Y(X)$ and $\pi_Y(Z)$ are far apart in $Y$, then the projections of $X$ and $Y$ to $Z$ are coarsely equal.  After slightly perturbing the projection distances $d^\pi_Y(X,Z)=\diam(\pi_Y(X)\cup \pi_Y(Z))$ to a distance $d_Y$, we can build the \emph{projection complex} $\mathcal P_K({\bf Y})$ for a fixed large constant $K$, which has vertices $Y\in\bf Y$ and an edge between $X, Z\in\bf Y$ if $d_Y(X,Z)$ is uniformly bounded for every $Y\in\mathbf{Y} \setminus \{X,Z\}$.  It is shown in \cite{bbf} that $\mathcal P_K(\bf Y)$ is a quasi-tree with a $G$--action.  From this quasi-tree, the \emph{quasi-tree of metric spaces} $\mc C_K({\bf Y})$ is formed by replacing each vertex labeled by $Y\in\bf Y$ with the space $Y$.

We now give the general construction from \cite{bbf}.  There are two main differences to keep in mind.  First, the ``projection" map which we will define does not have to have a geometric interpretation as an actual nearest-point projection; it will simply be a map satisfying certain axioms.  Second,  in general we have an index set $\mathbf Y$, and to each $Y\in\mathbf Y$ we associate a space $\mc C(Y)$.  In the example above, elements $Y$ of the index set were equated with the spaces $\mc C(Y)$.

Fix a set $\bf Y$, and for each $Y\in\bf Y$, let $\mc C(Y)$ be a geodesic metric space.  Let \[\pi_Y\colon {\bf Y}\setminus\{Y\}\to 2^{\mc C(Y)}\] be a function, which we call \emph{projection}.  
When $Y\neq X\neq Z$,  define a (pseudo-)distance function $d^\pi_Y$ by
 \[d^\pi_Y(X,Z)=\diam(\pi_Y(X)\cup\pi_Y(Z)).\]

For the rest of this section, assume that there is a constant $\theta\geq 0$ such that the following three conditions hold.
\begin{enumerate}[(P1)]
\setcounter{enumi}{-1}
\item The diameter $\diam \pi_X(Y)$ is uniformly bounded by $\theta$, independently of $X\in \mathbf{Y}$ and $Y\in \mathbf{Y}\setminus \{X\}$.
\item For any triple $X,Y,Z\in\bf Y$ of distinct elements, at most one of the three numbers \[d^\pi_X(Y,Z),\, d^\pi_Y(X,Z),\,d^\pi_Z(X,Y)\] is greater than $\theta$.
\item For any $X,Y\in\bf Y$, the set \[\{Z\in{\bf Y} \setminus \{X,Y\} : d^\pi_Z(X,Y)>\theta\}\] is finite.
\end{enumerate}

We will  modify these distances  by a bounded amount.  We first need a definition.
\begin{defn}
For $X,Z\in\bf Y$ with $X\neq Z$, let $\mathcal H(X,Z)$ be the set of pairs $(X',Z')\in\bf Y\times\bf Y$ with $X'\neq Z'$ such that one of the following four conditions holds:
\begin{itemize}
\item both $d^\pi_X(X',Z'),d^\pi_Z(X',Z')>2\theta$;
\item $X=X'$ and $d^\pi_Z(X,Z')>2\theta$;
\item $Z=Z'$ and $d^\pi_X(X',Z)>2\theta$;
\item $(X',Z')=(X,Z)$.
\end{itemize}
\end{defn}
Define the modified distance functions \[d_Y\colon ({\bf Y}\setminus\{Y\})\times({\bf Y}\setminus\{Y\})\to[0,\infty)\] 
by \[d_Y(X,Z)=\begin{cases} 0 & \textrm{if $Y$ is contained in a pair in $\mc H(X,Z)$} \\ \inf_{(X',Z')\in\mc H(X,Z)}d^\pi_Y(X',Z') & \textrm{else.}\end{cases}\] It is immediate from the definition that the modified distance functions satisfy $d_{Y}\leq d^\pi_{Y}$ for all $Y\in\bf Y$.
Suppose (P0)--(P2) are satisfied by $({\bf Y},\theta,\{d^\pi_Y\})$,  fix $K\geq \Theta$ where $\Theta=\Theta(\theta)$ is the constant from \cite[Theorem~3.3]{bbf}, and let ${\bf Y}_K(X,Z)=\{Y\in {\bf Y}: d_Y(X,Z)>K\}$.     We construct a space $\mc P_K(\bf Y)$ as follows.  
\begin{defn}
The \emph{projection complex} $\mc P_K(\bf Y)$ is the following graph.  The vertex set of $\mc P_K(\bf Y)$ is $\bf Y$.  Two distinct vertices $X$ and $Z$ are connected with an edge if ${\bf Y}_K(X,Z)=\emptyset$. Denote the distance function for this graph by $d(\cdot,\cdot)$.
\end{defn}

\begin{thm}[{\cite[Theorem~3.16]{bbf}}]
For $K$ sufficiently large, $\mc P_K(\bf Y)$ is a quasi-tree.
\end{thm}

We are now ready to give the construction of the quasi-tree of metric spaces. Fix a constant $L=L(K)$ as in \cite[Lemma~4.2]{bbf}.

\begin{defn}
A \emph{quasi-tree of metric spaces} is the path metric space $\mc C_K({\bf Y})$ obtained by taking the disjoint union of the metric spaces $\mc C(Y)$ for $Y\in{\bf Y}$, and, if $d(X,Z)=1$ in $\mc P_K({\bf Y})$, we attach an edge of length $L$ from every point in $\pi_X(Z)$ to every point in $\pi_Z(X)$.
\end{defn}

\begin{thm}[{\cite[Theorem~A]{bbf}}]
\label{thm:isomembed}
Suppose $\bf Y$ is a collection of geodesic metric spaces and for every $X,Y\in\bf Y$ with $X\neq Y$ we are given a subset $\pi_X(Y)\subset \C(X)$ such that (P0)-(P2) hold and $K$ is sufficiently large. Then the spaces $\C(X)$ for $X\in \bf Y$ are isometrically embedded in $\C_K(\bf Y)$. Moreover, for each distinct $X,Y\in \bf Y$, the nearest point projection of $\C(Y)$ to $\C(X)$ in $\C_K(\bf Y)$ is a uniformly bounded set uniformly close to $\pi_X(Y)$.
\end{thm}

The following is straightforward to verify:

\begin{thm}
Suppose that $G$ is a group which acts on the set $\bf{Y}$ such that for each $X\in \bf{Y}$ there is an isometry $F_g^X:\C(X)\to \C(g(X))$ and the isometries $F_g^X$ satisfy:
\begin{itemize}
\item if $g,h\in G$ and $X\in \bf{Y}$, then $F_h^{g(X)}\circ F_g^X=F_{hg}^X$; and
\item if $X,Y\in \bf{Y}$, then $F_g^Y(\pi_Y(X))=\pi_{g(Y)}(g(X))$.
\end{itemize}
Then there is an induced action of $G$ on $\C_K(\bf{Y})$ by isometries.
\end{thm}

\noindent We will frequently denote the isometry $F_g^X$ simply by $g$.

The quasi-trees of metric spaces $\C(\mathbf{Y})$ have nice geometric properties when the geodesic metric spaces $\C(Y)$ have these properties \textit{uniformly}. To state the next theorem, recall Manning's \textit{bottleneck criterion} (\cite[Theorem~4.6]{pseudochar}): the geodesic metric space $X$ is a quasi-tree if and only if there exists $\Delta\geq 0$ with the following property. Let $x,y\in X$, let $\gamma$ be a geodesic from $x$ to $y$, and let $z$ be the midpoint of $\gamma$. Then any continuous path from $x$ to $y$ passes through the $\Delta$--neighborhood of $z$. The constant $\Delta$ is called a \textit{bottleneck} constant for $X$.

\begin{thm}[{\cite[Theorem~4.14]{bbf}}]
\label{thm:bbfqtree}
Suppose that all $\C(Y)$ for $Y\in \mathbf{Y}$ are quasi-trees with a uniform bottleneck constant $\Delta$. Then $\C_K(\mathbf{Y})$ is a quasi-tree for $K$ large enough.
\end{thm}

%%%%%%%%%%%%%%%%%%%%%%%%%%%%%%%%%%%%%%%%%%%%%%%%%%%%%%%%%%%%%%%%%%%%%%%%%%%%%%%%%%%%%%%%%%%%%%%%%%%%%%%%%%%%%%%%%%%%%%%%%%%%%%%%%%%%%%%%%%%%%%%%%%%%%%%

\section{Main lemma and first applications} \label{sec:mainlem}

In this section we prove our main lemma, Lemma \ref{lem:mainlem}, and give several relatively straightforward applications. Recall the statement:

\mainlem*

\begin{proof}
Suppose that $G\curvearrowright X\preceq G\curvearrowright Z$ and $G\curvearrowright Y\preceq G\curvearrowright Z$. Then $a$ and $b$ are both loxodromic in the action $G\curvearrowright Z$. Since $a$ and $b$ commute, their fixed points on $\partial Z$ are the same.

Let $f\colon Z\to X$ be a $K$--coarsely Lipschitz, coarsely $G$--equivariant map. Choose a base point $z\in Z$. Then there exists $D>0$ such that $d_X(f(gz),gf(z))\leq D$ for any $g\in G$. The sequences $\{a^nz\}_{n\in \Z}$ and $\{b^n z\}_{n\in \Z}$ are quasigeodesics with the same pair of endpoints on $\partial Z$. Hence they are $E$--Hausdorff close for some $E>0$. 

The set $S=\{b^n f(z)\}_{n\in \Z}$ is bounded since $b$ acts elliptically on $X$. Hence there exists $N$ large enough that $d_X(a^nf(z),S)> KE+K+2D$ for all $n\geq N$. However, given any $n\geq N$, there exists some $m\in \Z$ with $d_Z(b^mz,a^nz)\leq E$. We then have \[d_X(b^mf(z),a^nf(z))\leq d_X(f(b^mz),f(a^nz))+2D\leq Kd_Z(b^mz,a^nz)+K+2D\leq KE+K+2D.\] This is a contradiction.
\end{proof}

%%%%%%%%%%%%%%%%%%%%%%%%%%%%%%%%%%%%%%%%%%%%%%%%%%%%%%%%%%%%%%%%%%%%%%%%%%%%%%%%%%%%%%%%%%%%%%%%%%%%%%%%%%%%%%%%%%%%%%%%%%%%%%%%%%%%%%%%%%%%%%%%%%%%%%%

%\section{Right-angled Artin groups, Baumslag-Solitar groups, and solvable groups}
%In this section, we give three straightforward applications of Lemma \ref{lem:mainlem}.
\subsection{Right-angled Artin groups}
\label{sec:raags}
We first prove a more general theorem about groups which admit retracts onto subgroups isomorphic to $\Z^2$ and then use this to prove Theorem \ref{mainthm} for the special case of right-angled Artin groups.

\begin{thm}\label{thm:retract}
Let $G$ be a group which admits a retract onto a subgroup isomorphic to $\Z^2$.  Then $\H(G)$ contains no largest element.
\end{thm}

\begin{proof}
Let $r\colon G\to\mathbb Z^2=\langle a,b\rangle$ be a retract.  We may define a projection $p\colon \langle a,b\rangle\to \R$ by $p(a)=1$ and $p(b)=0$ and then define an action of $G$ on $\R$ by translations via \[g(x)=x+p(r(g)) \text{ for }g\in G \text{ and } x \in \R.\] In this action, $a$ is loxodromic while $b$ is elliptic since $p(r(a))=1$ and $p(r(b))=0$. Similarly, we define an action on $\R$ by translations where $a$ acts elliptically and $b$ acts loxodromically. By Lemma \ref{lem:mainlem} this completes the proof.
\end{proof}

Recall that given a finite simplicial graph $\Gamma$, the right-angled Artin group $A(\Gamma)$ is defined by the presentation \[A(\Gamma)=\langle v \in V(\Gamma) : [v,w]=1 \textit{ if $v$ and $w$ are joined by an edge in } \Gamma\rangle\] where $V(\Gamma)$ denotes the set of vertices of $\Gamma$. The group $\Gamma$ is free if and only if $\Gamma$ has no edges.

\begin{cor}
Let $G$ be a right-angled Artin group which is not free. Then $\H(G)$ contains no largest element.
\end{cor}

\begin{proof}
Let $a$ and $b$ be generators of $G$ corresponding to any two vertices of the defining graph $\Gamma$ which are joined by an edge. Then there is a retract $r\colon G\to\mathbb Z^2=\langle a,b\rangle$ defined by fixing $a$ and $b$ and sending all other generators to the identity in $\langle a,b \rangle$. Applying Theorem \ref{thm:retract} completes the proof.%We may define a projection $p\colon \langle a,b\rangle\to \R$ by $p(a)=1$ and $p(b)=0$ and then define an action of $G$ on $\R$ by translations by \[g(x)=x+p(r(g)) \text{ for }g\in G \text{ and } x \in \R.\] In this action, $a$ is loxodromic while $b$ is elliptic since $p(r(a))=1$ and $p(r(b))=0$. Similarly, we define an action on $\R$ by translations where $a$ acts elliptically and $b$ acts loxodromically. By Lemma \ref{lem:mainlem} this completes the proof.
%Let $G$ be a right-angled Artin group, whose graph contains an edge.  Then there is a retract $r\colon G\to\mathbb Z^2=\langle a,b\rangle$.  $\mathbb Z^2$ clearly has two lineal actions, coming from $\mathbb Z^2/\langle a\rangle\simeq \langle b\rangle$ and $\mathbb Z^2/\langle b\rangle\simeq \langle a\rangle$ acting on their Cayley graphs.  We use the retraction to define actions $G\curvearrowright\langle b\rangle$ and $G\curvearrowright \langle a\rangle$.  As these actions are linear, they are either equal or incomparable in $\mathcal H(G)$.  Since $a$ is loxodromic in the first but not the second action, it is clear that they are not comparable. However, since $a$ and $b$ commute, there can be no element of $\mathcal H(G)$ which dominates both of these actions.
\end{proof}

%\begin{rem}
%I believe this theorem should hold as written for all special groups.  Can we extend it to all virtually special groups?  If $G$ is virtually special, let $G'$ be the finite index special subgroup of $G$.   Then we have two (oriented -- i.e. no element of $G$ swaps the endpoints) lineal action of $G'$ in which two different commuting elements are loxodromic.  Can we extend the Busemann quasimorphism for each of these actions to all of $G$?  I believe the Busemann quasimorphism can be constructed whenever there is an action $G\acts X$ on a hyperbolic space such that $G$ fixes a point in $\partial X$.  (The two lineal actions are oriented because the action of $\Z^2$ is.). If this is correct, then we can always extend lineal actions of finite index subgroups to lineal actions of the entire group with the same set of loxodromic elements.
%\end{rem}

\subsection{Baumslag-Solitar groups}
\label{sec:bs}
Let $m,n\in\Z \setminus \{0\}$.  The \textit{Baumslag-Solitar group} is defined as  $BS(m,n)=\langle a,b : ba^mb^{-1}=a^n\rangle$.

We will use the following in our proof of Theorem \ref{mainthm} for Baumslag-Solitar groups:

\begin{lem}
\label{lem:distorted}
Let $G$ be a finitely generated group, and let $a\in G$ be distorted (that is, the inclusion of the subgroup $\langle a \rangle$ generated by $a$ in $G$ is not a quasi-isometric embedding). Then in any cobounded action $G\curvearrowright X$ with $X$ hyperbolic, the element $a$ is not loxodromic.
\end{lem}

\begin{proof}
Using the Schwarz-Milnor Lemma \cite[Lemma~3.11]{abo} we may suppose without loss of generality that $X$ is the Cayley graph of $G$ with respect to a generating set $T$.

Let $S$ be a \textit{finite} generating set for $G$. Then for any $C>0$ there exists $n>0$ such that the word length $\|a^n\|_S<Cn$. Consequently, we may write $a^n=g_1\ldots g_k$ where $g_1,\ldots,g_k\in S$ and $k\leq Cn$. Write $M=\max \{ \|h\|_T : h\in S\}$, which exists because $S$ is finite. Therefore we also have \[\|a^n\|_T\leq \|g_1\|_T +\cdots+\|g_k\|_T< k M\leq CMn.\] Consequently, for any $D>0$ there exists $n>0$ such that $\|a^n\|_T<Dn$. We have that $\|a^n\|_T$ is the distance from $1$ to $a^n\cdot 1$ in the Cayley graph $\Gamma(G,T)$, and therefore these distances do not grow linearly with $n$. This proves that $a$ is not loxodromic, as claimed.
\end{proof}

\begin{lem}
Let $n\in\Z\setminus \{0\}$. Then $BS(1,n)$ is solvable.
\end{lem}

\begin{proof}
One checks that the  subgroup normally generated by $a$, denoted $ \llangle a  \rrangle$,  is generated by the conjugates $b^rab^{-r}$ for $r\in \Z$. Furthermore, $\llangle  a \rrangle $ is isomorphic to $\Z\left[\frac{1}{|n|}\right]$ via the the homomorphism $\llangle  a \rrangle \to \Z\left[\frac{1}{|n|}\right]$ defined on the generating set by $b^rab^{-r}\mapsto |n|^r$. We then see that $BS(1,n)$ admits an isomorphism \[BS(1,n)=\llangle  a  \rrangle \rtimes \langle b \rangle \cong \Z\left[\frac{1}{|n|}\right] \rtimes \Z\] where the generator $t$ of $\Z$ acts on $\Z\left[\frac{1}{|n|}\right]$ by $t:x\mapsto nx$. Clearly then $BS(1,n)$ is solvable, as claimed.
%Consider the normal subgroup generated by $a$, $\langle\langle a \rangle\rangle$. Note that $BS(1,n)/\langle\langle a\rangle\rangle$ is isomorphic to $\langle b\rangle\cong \Z$ and therefore is abelian.
%
%To show that $BS(1,n)$ is abelian it suffices to show that $\langle\langle a\rangle\rangle$ itself is abelian.
%
%First note that $\langle\langle a\rangle\rangle$ is generated by the elements $b^rab^{-r}$ for $r\in \Z$. To see this, we claim that $b^rab^{-r}$ commutes with $a$ for any $r\in \Z$. This will prove the statement, as the effect of conjugating $a$ by a word in $a$ and $b$ will simply be to successively conjugate by the powers of $b$ appearing in the word, thus leaving us with one of the elements $b^r a b^{-r}$. To prove the claim, note that if $r\geq 0$ then $b^rab^{-r}=a^{n^r}$ clearly commutes with $a$. On the other hand, if $r\leq 0$ then $(b^rab^{-r})^{n^r}=b^ra^{n^r}b^{-r}=a$ and thus since $a$ is a power of $b^rab^{-r}$ it clearly commutes with $b^rab^{-r}$.
%
%To show that $\langle\langle a\rangle\rangle$ itself is abelian it suffices to check that two of the conjugates $b^kab^{-k}$ and $b^lab^{-l}$ commute. Suppose without loss of generality that $l>k$. We compute: \[ (b^kab^{-k})(b^lab^{-l})=b^ka(b^{l-k}a)b^{-l}=b^ka(a^{n^{l-k}}b^{l-k})b^{-l}=b^ka^{n^{l-k}+1}b^{-k}\] whereas \[(b^lab^{-l})(b^kab^{-k})=b^l(ab^{-(l-k)})ab^{-k}=b^l(b^{-(l-k)}a^{n^{l-k}})ab^{-k}=b^ka^{n^{l-k}+1}b^{-k}.\]
\end{proof}

\begin{thm}
\label{thm:bs}
Let $m,n\in \Z\setminus \{0\}$. Then $\H(BS(m,n))$ contains no largest element.
\end{thm}

\begin{proof}
Note that $BS(m,n)\cong BS(n,m)$ via the map $a\mapsto a, b\mapsto b^{-1}$. Hence we may suppose without loss of generality that $|m|\leq |n|$. Moreover, we have $BS(m,n)\cong BS(-m,-n)$ via the map $a\mapsto a^{-1}, b\mapsto b$. Therefore we may suppose without loss of generality that $m\geq 1$. By these remarks it suffices to consider three cases.  In all that follows, set $G=BS(m,n)$ with $m,n$ depending on the particular case, as described.

\begin{enumerate}[(1)]
\item \underline{$m=|n|$.}

In this case we show that there are cobounded hyperbolic actions $G\curvearrowright X$ with $a$ acting loxodromically and $b$ acting elliptically and $G\curvearrowright Y$ with $a$ acting elliptically and $b$ acting loxodromically. Then we apply Lemma \ref{lem:mainlem}.

The action $G\curvearrowright Y$ is the lineal action corresponding to the homomorphism $G\to \Z$ defined by $a\mapsto 0$, $b\mapsto 1$. We let $\Z$ act on $\R$ by translation and thus define an action of $G$ on $\R$.

If $n=m$, then the action $G\curvearrowright X$ is the lineal action corresponding to the homomorphism $G\to \Z$ defined by $a\mapsto 1, b\mapsto 0$.  We again let $\Z$ act on $\R$ by translation. If $n=-m$, then the action $G\curvearrowright X$ is the lineal action corresponding to the homomorphism $G\to D_\infty=\langle t,r : rtr^{-1}=t^{-1}\rangle$ defined by $a\mapsto t, b\mapsto r$. Here, we let $D_\infty$ act on $\R$ by $t(x)=x+1$ and $r(x)=-x$ for $x\in \R$.

\item \underline{$1=m<|n|$.}

In this case we consider two cobounded hyperbolic actions $G\curvearrowright \hyp^2$ and $G\curvearrowright T$, where $T$ is the Bass-Serre tree of the HNN extension $G\cong \langle a\rangle *_{\langle a \rangle= \langle a^n \rangle}$ (this corresponds to the expression of $G$ as a one edge graph of groups with vertex group $\langle a \rangle$).

We consider the upper half plane model of $\mathbb{H}^2$ with orientation-preserving isometry group $\text{PSL}(2,\R)$ acting by M\"{o}bius transformations. If $n>0$, then $G\curvearrowright \mathbb{H}^2$ is given by \[a\mapsto \begin{pmatrix} 1 & 1 \\ 0 & 1 \end{pmatrix},\quad b\mapsto \begin{pmatrix} \sqrt{n} & 0 \\ 0 & 1/\sqrt{n} \end{pmatrix}.\] If $n<0$, then $G\curvearrowright \mathbb{H}^2$ is given by \[a\mapsto \begin{pmatrix} 1 & 1 \\ 0 & 1 \end{pmatrix}, \quad b\mapsto \psi \circ \begin{pmatrix} \sqrt{|n|} & 0 \\ 0 & 1/\sqrt{|n|} \end{pmatrix},\] where $\psi$ is the orientation-reversing isometry of $\mathbb{H}^2$ consisting of reflection across the positive imaginary axis (i.e., $\psi(z)=-\overline{z}$ for $z\in \hyp^2$).

Note that in the action $G\curvearrowright \mathbb{H}^2$ every conjugate of $b$ has a common attracting fixed point, but that the various conjugates of $b$ have different repelling fixed points. In contrast, in the action $G\curvearrowright T$ every conjugate of $b$ has a common repelling fixed point, but the various conjugates have different attracting fixed points. Hence if there is a hyperbolic action $G\curvearrowright Z$ larger than both $G\curvearrowright \mathbb{H}^2$ and $G\curvearrowright T$, then the action $G\curvearrowright Z$ is general type --- every conjugate of $b$ acts loxodromically, but there are different conjugates of $b$ with disjoint fixed point sets on $\partial Z$. By the Ping-Pong Lemma, $G$ contains a free group. However, this is a contradiction, as $G$ is solvable.

\item \underline{$1<m<|n|$.}

In this case we consider two cobounded hyperbolic actions $G\curvearrowright \mathbb{H}^2$ and $G\curvearrowright T$ where $T$ is the Bass-Serre tree corresponding to the HNN extension $G\cong \langle a\rangle *_{\langle a^m \rangle= \langle a^n \rangle}$.

The action $G\curvearrowright \mathbb{H}^2$ is given by \[a\mapsto \begin{pmatrix} 1 & 1 \\ 0 & 1 \end{pmatrix}, \quad b\mapsto \begin{pmatrix} \sqrt{n/m} & 0 \\ 0 & \sqrt{m/n} \end{pmatrix}\] if $n>0$ and by \[a\mapsto \begin{pmatrix} 1 & 1 \\ 0 & 1 \end{pmatrix}, \quad b\mapsto \psi \circ \begin{pmatrix} \sqrt{|n|/m} & 0 \\ 0 & \sqrt{m/|n|} \end{pmatrix}\] if $n<0$, where $\psi$ is the orientation-reversing isometry $z\mapsto -\overline{z}$ defined earlier.

Note that the action on $T$ is general type. To see this, in the action of $G$ note that $a$ fixes a vertex $v$ along the axis of $b$. At $v$ there are $|n|$ outgoing edges and $m$ incoming edges and $a$ freely permutes the outgoing edges and freely permutes the incoming edges. Consequently $aba^{-1}$ has an axis which passes through $v$ and enters $v$ through a \textit{different} incoming edge than the axis of $b$ and exits $v$ through a \textit{different} outgoing edge than the axis of $b$. Consequently $b$ and $aba^{-1}$ have disjoint fixed point sets on $\partial T$.

We suppose again that there is a hyperbolic action $G\curvearrowright Z$ which dominates both of these. Since $G\curvearrowright T$ is general type, $G\curvearrowright Z$ must be general type. As $a$ acts parabolically in the action $G\curvearrowright \mathbb{H}^2$, it must act parabolically or loxodromically in the action $G\curvearrowright Z$. Since $\langle a\rangle$ is distorted in $G$, $a$ must in fact act parabolically in $G\curvearrowright Z$ by Lemma~\ref{lem:distorted}. Moreover, the equation $ba^mb^{-1}=a^n$ implies that $b$ fixes the single fixed point of $a$ on $\partial Z$. Thus every element of $G$ fixes this point, and we obtain a contradiction to $G\curvearrowright Z$ being general type.\qedhere
\end{enumerate}
\end{proof}

\begin{rem}
When $1=m<n$, the poset of hyperbolic structures of $BS(1,n)$ has been completely described in \cite{ar}, and Theorem \ref{thm:bs} follows in this case (see \cite[Corollary~1.2]{ar}).
\end{rem}

\subsection{Solvable groups}
\label{sec:solvable}

\begin{thm}
Let $G$ be a finitely generated solvable group with abelianization of rank at least two. Then $\H(G)$ contains no largest element.
\end{thm}

\begin{proof}
The abelianization $G/[G,G]$ is isomorphic to $\Z^n \times F$ where $F$ is a finite abelian group and $n\geq 2$.  Let $f\colon G\to \Z^n \times F$ be the abelianization map and $p_1\colon \Z^n\times F\to \Z$ and $p_2\colon \Z^n\times F\to \Z$ be the projections to the first and second factors of $\Z^n$, respectively. 

We may choose $a,b\in G$ with $p_1( f(a))=1$ and $p_1( f(b))=0$ and $p_2( f(a))=0$ and $p_2( f(b))=1$. We obtain actions $G\curvearrowright \R$ by \[g\colon x\mapsto x+p_1( f(g)) \text{ and } g\colon x\mapsto x+p_2( f(g))\] for $g\in G$ and $x\in \R$. We denote these actions by $G\curvearrowright X_1$ and $G\curvearrowright X_2$, respectively.

Suppose there exists $G\curvearrowright Z$ with $G\curvearrowright X_1 \preceq G\curvearrowright Z$ and $G\curvearrowright X_2 \preceq G\curvearrowright Z$. Note that  $a$ is loxodromic and $b$ is elliptic in $G\curvearrowright X_1$ and $a$ is elliptic and $b$ is loxodromic in $G\curvearrowright X_2$. Thus in the action $G\curvearrowright Z$ both $a$ and $b$ must be loxodromic. Since $G$ is solvable, it contains no free subgroup. Moreover, by the Ping-Pong Lemma, any high enough powers of independent loxodromic elements of $G$ in the action $G\curvearrowright Z$ generate a free subgroup of $G$. Hence $a$ and $b$ are not independent in this action.

Up to replacing one of $a$ or $b$ by its inverse, we may suppose that $a$ and $b$ have the same attracting fixed point $p\in \partial Z$. Fix a basepoint $z\in Z$. Then $\{a^n z\}_{n \in \Z_{\geq 0}}$ and $\{b^n z\}_{n\in \Z_{\geq 0}}$ are quasigeodesic rays with the same endpoint $p\in \partial Z$. Hence there exists $E>0$ such that the rays are eventually $E$--close to each other. It follows that there exists $N>0$ such that for all $n\geq N$, there exists $m\in \Z_{\geq 0}$ with $d(a^nz,b^mz)\leq E$. As in the proof of Lemma \ref{lem:mainlem}, this contradicts that $G\curvearrowright Z\succeq G\curvearrowright X_1$.
\end{proof}

\section{Mapping class groups}
\label{sec:mcgs}

Let $S$ be an orientable surface of genus $g$ with $n$ punctures. We define the \textit{complexity} of $S$ to be $\xi(S)=3g-3+n$. The \textit{mapping class group of $S$} is the group $\Mod(S)$ of orientation-preserving homeomorphisms of $S$ up to isotopy.

\subsection{Largest actions}
The main result of this section is the following:

\begin{thm}
\label{thm:mcgaccess}
Suppose that $\xi(S)\geq 2$. Then $\H(\Mod(S))$ contains no largest element.
\end{thm}

\begin{rem}
The condition $\xi(S)< 2$ turns out to be equivalent to $\Mod(S)$ being a \textit{hyperbolic} group. So we actually have the following classification: $\H(\Mod(S))$ contains a largest element if and only if $\xi(S)<2$.
\end{rem}

In the next subsection, we will prove a finer theorem about the structure $\H(\Mod(S))$ when $S$ is a closed surface of genus at least two, which will also imply Theorem \ref{thm:mcgaccess} when $S$ is closed. In this section, we prove Theorem \ref{thm:mcgaccess} using Lemma \ref{lem:mainlem}. 

Our main tool is the following lemma, which is a corollary of \cite[Proposition~4.3]{scl}.  We will first use this lemma to prove Theorem \ref{thm:mcgaccess}, and then we will give an outline of the proof of the lemma in order to preview some of the machinery that will be developed for flip graph manifold groups.  Given an essential simple closed curve $\gamma$, we denote by $T_\gamma$  the Dehn twist about $\gamma$.

\begin{lem}
\label{lem:qmseparation}
Let $\alpha$ and $\beta$ be two essential simple closed curves on $S$ which lie in different $\Mod(S)$--orbits. There exist quasimorphisms $q\colon \Mod(S)\to \R$ and $q'\colon \Mod(S)\to \R$ such that 
\begin{itemize}
\item $q(T_\alpha)\neq 0$ and $q(T_\beta)=0$, and
\item $q'(T_\alpha)=0$ and $q'(T_\beta)\neq 0$.
\end{itemize}
\end{lem}

\begin{proof}[Proof of Theorem \ref{thm:mcgaccess} using Lemma \ref{lem:qmseparation}]
Suppose first that $S$ is not the five-times punctured sphere. Since $\xi(S)\geq 2$, there exist simple closed curves $\alpha$ and $\beta$ in $S$ which lie in different $\Mod(S)$ orbits. Namely, if $S$ has genus zero, then we may take $\alpha$ to be a curve bounding a twice-punctured disk and $\beta$ to be a curve bounding a thrice-punctured disk. Otherwise we may choose $\alpha$ to be nonseparating and $\beta$ to be separating. Moreover, we may choose $\alpha$ and $\beta$ to be disjoint, so that $T_\alpha$ and $T_\beta$ commute.

By Lemmas \ref{lem:quasiline} and \ref{lem:qmseparation}   we obtain an action $\Mod(S)\curvearrowright X$ where $X$ is a quasi-line and $T_\alpha$ acts loxodromically and $T_\beta$ acts elliptically. Similarly we obtain an action $\Mod(S)\curvearrowright Y$ where $Y$ is a quasi-line and $T_\alpha$ acts elliptically while $T_\beta$ acts loxodromically. Applying Lemma \ref{lem:mainlem} completes the proof.

Now we suppose that $S$ is the five-times punctured sphere. In this case, there is only a single $\Mod(S)$--orbit of essential simple closed curves. Choose $\alpha$ to be an essential simple closed curve; it bounds a three-times punctured disk $V$. We will choose $\phi$ to be a pseudo-Anosov on $V$ so that $T_\alpha$ and $\phi$ commute. We will then find homogeneous quasimorphisms $q$ and $q'$ such that $q(T_\alpha)=0$, $q(\phi)\neq 0$ and $q'(T_\alpha)\neq 0$, $q'(\phi)= 0$. However we must be careful to choose $\phi$ to be \textit{chiral} (see \cite{scl}). Recall that $\phi$ is chiral if $\phi^n$ is not conjugate to $\phi^{-n}$ in $\Mod(S)$ for any $n\neq 0$.

To choose $\phi$ we argue as follows. The mapping class group of the three-times punctured disk $V$ is the braid group $B_3$ on three strands. There is a surjective homomorphism $F\colon\Mod(V)\to \SL(2,\Z)$ with kernel generated by the Dehn twist along $\partial V=\alpha$, defined as follows. The group $B_3$ is generated by two half twists $\sigma$ and $\tau$, which satisfy the braid relation $\sigma\tau\sigma=\tau\sigma\tau$. We define \[F(\sigma)=\begin{pmatrix} 1 & 1 \\ 0 & 1 \end{pmatrix}, F(\tau)= \begin{pmatrix}1 & 0 \\ -1 & 1 \end{pmatrix}.\] We see that if $F(\psi)$ is an \textit{Anosov matrix} (i.e. a matrix with two distinct real eigenvalues) then $\psi$ is pseudo-Anosov. This holds since any reducible element $\eta$ of $\Mod(V)$ is conjugate to the product of a power of a Dehn twist on $\partial V$ and a power of some half twist, and therefore $F(\eta)$ is unipotent.

If $\phi\in \Mod(V)$ has the property that $F(\phi)$ is Anosov and $F(\phi)^n$ is not conjugate to $F(\phi)^{-n}$ for any $n\neq 0$, then $\phi$ is pseudo-Anosov (possibly twisting along $\partial V=\alpha$) and $\phi^n$ is not conjugate to $\phi^{-n}$ in $\Mod(V)$ for any $n\neq 0$. Moreover, we see that if $g\in \Mod(S)$ conjugates a non-trivial power $\phi^n$ to $\phi^{-n}$ then $g$ must fix $V$ and therefore restrict to an element of $\Mod(V)$ which conjugates $\phi^n$ to $\phi^{-n}$. This is a contradiction.

Matrices $A\in \SL(2,\Z)$ with the property that $A^n$ is not conjugate to $A^{-n}$ for any $n\neq 0$ do exist. See \cite{reversing} (in particular Example 2 and Lemma 9). Therefore, we may choose $\phi \in \Mod(V)$ such that $F(\phi)$ is Anosov and $F(\phi)^n$ is not conjugate to $F(\phi)^{-n}$ for any $n\neq 0$. It follows that $\phi$ is pseudo-Anosov and chiral. By \cite[Proposition~4.3]{scl}, there exist homogeneous quasimorphisms $q$ and $q'$ with $q(T_\alpha)=0$, $q(\phi)\neq 0$ and $q'(T_\alpha)\neq 0$, $q'(\phi)= 0$. Applying Lemmas \ref{lem:quasiline} and \ref{lem:mainlem} completes the proof.
\end{proof}

In order to preview some of the machinery that will be developed for flip graph manifold groups, we outline the proof of Lemma \ref{lem:qmseparation}. Details may be found in \cite{scl}.

\begin{proof}[Outline of proof of Lemma \ref{lem:qmseparation}]
In \cite{bbf}, Bestvina--Bromberg--Fujiwara construct an action of a finite-index subgroup $\Gamma< G$ on a quasi-tree $\mc C(\bf X)$ in which $T_\alpha$ is loxodromic (\cite[Theorem~5.9]{bbf}).  We briefly review the construction here.

The curve graph $\C(\gamma)$ of a simple closed curve $\gamma$, as defined in \cite[Section~5]{bbf}, is quasi-isometric to $\R$, and $T_\gamma$ acts on it as a loxodromic isometry. We would like to define a quasi-tree of metric spaces $\C(\mathbf{Y})$ where $\mathbf{Y}$ is the collection of all curves in the $\Mod(S)$--orbit of $\gamma$ and if $\delta\in \mathbf{Y}$ then $\C(\delta)$ is the curve graph of $\delta$. However, this will not work, because such a collection $\mathbf{Y}$ contains disjoint elements, making it impossible to define subsurface projections between the elements of $\mathbf{Y}$.

Instead, we choose  $\mathbf{Y}$ to be a subset of the curves in the $\Mod(S)$--orbit of $\gamma$. To do this, Bestvina--Bromberg--Fujiwara construct in \cite[Lemma 5.6]{bbf} a specific coloring of the (isotopy classes of) subsurfaces of $S$ with finitely many colors. This coloring has the property that disjoint subsurfaces have distinct colors. By the proof of \cite[Lemma~5.7]{bbf}, $\Mod(S)$ permutes the set of colors of subsurfaces and thus there is a finite index normal subgroup $\Gamma\leq \Mod(S)$ that preserves the colors. Now we may choose $\mathbf{Y}$ to be the set of curves in $\Mod(S) \cdot\gamma$ with \textit{the same color} as $\gamma$.

By machinery for mapping class groups developed in \cite{hier} and \cite{asymp},  the axioms (P0)--(P2) are satisfied for $\mathbf{Y}$ and subsurface projections $\pi_Y$ between the elements of $\mathbf{Y}$. Hence we obtain a quasi-tree of metric spaces $\C_K(\mathbf{Y})$ whenever $K$ is large enough. Although $\Mod(S)$ does not act on $\C_K(\mathbf{Y})$, the color-preserving subgroup $\Gamma$ \textit{does} act on $\C_K(\mathbf{Y})$.

Set $\gamma=\alpha$, where $\alpha$ is as in the statement of Lemma \ref{lem:qmseparation}, and set $\mathbf{Y}$ to be the set of elements of $\Mod(S)\cdot \gamma$ with the same color as $\gamma$, as above. We will use the action of $\Gamma$ on $\C_K(\mathbf{Y})$ to define a quasimorphism $\Mod(S)\to \R$.  Let $k$ be the index of $\Gamma$ in $G$.
We have that
\begin{itemize}
\item $T_\alpha^k\in \Gamma$,
\item $T_\alpha^k$ acts loxodromically on $\C_K(\mathbf{Y})$, and
\item $T_\alpha^k$ is WWPD in the action $\C_K(\mathbf{Y})$.
\end{itemize}
The last point follows easily from the quasi-tree of metric spaces machinery. In fact it is straightforward to check that $T_\alpha^k$ is WWPD${}^+$. As in Proposition \ref{prop:qmconstruction}, we define a homogeneous quasimorphism $q_0\colon\Gamma\to \R$ such that $q_0(T_\alpha^k)\neq 0$. Bestvina--Bromberg--Fujiwara use the construction of $q_0$ to show that $q_0(T_\beta^k)=0$.

Choose $h_1,\ldots,h_k$ to be coset representatives of $\Gamma$ in $\Mod(S)$. We first modify $q_0$  by defining \[q_0'(g)=\sum_{i=1}^kq_0(h_i^{-1}gh_i)\] for $g\in \Gamma$. This modified $q_0'$ satisfies $q_0'(hgh^{-1})=q_0'(g)$ for any $h\in \Mod(S)$. Furthermore, it extends to a homogeneous quasimorphism $q\colon\Mod(S)\to \R$ by setting $q(g)=\frac{1}{k} q_0'(g^k)$ (see \cite[Section~7]{symmetric}). Bestvina--Bromberg--Fujiwara also show that $q(T_\alpha)\neq 0$ and $q(T_\beta)=0$, as desired.

The existence of $q'$ follows immediately from the existence of $q$, since we only required $\alpha$ and $\beta$ to lie in distinct mapping class group orbits.
\end{proof}

\subsection{Maximal lineal actions}
In this section, we prove an extension of Theorem \ref{thm:mcgaccess} in the case that $S$ has no punctures.  Let $\Mod(S)$ be the mapping class group of a closed surface $S$, and consider a proper, connected, essential  subsurface $V$ of $S$. Suppose moreover that $V$ is disjoint from some element of its mapping class group orbit; that is, there is $h\in \Mod(S)$ such that $hV$ and $V$ have disjoint representatives in their isotopy classes. We show that under a certain technical condition on $V$, if $[\Mod(S)\curvearrowright X]$ is a hyperbolic structure  and  there exists $\phi\in \Mod(S)$ supported on $V$ that acts loxodromically on $X$, then the structure must in fact be lineal.

Before stating this result precisely, we introduce some notation.  Given two isotopy classes of essential subsurfaces $A,B$ of $S$, we write $A\perp B$ if $A$ and $B$ have disjoint representatives and $A\nest B$ if $A$ has a representative contained in $B$.   Let $g(T)$ denote the genus of a finite type surface $T$, let $b(T)$ denote the number of boundary components (or punctures), and let $\xi(T)=3g(T)-3+b(T)$ denote the complexity. We note that complexity is monotonic under inclusion; that is, if $A\nest B$, then $\xi(A)\leq \xi(B)$.

We may write $S\setminus V=W_1 \sqcup W_2 \sqcup \ldots \sqcup W_n$ where $W_i$ is a closed subsurface with $b(W_i)$ boundary components, isotopic to $b(W_i)$ of the boundary components of $V$. Since there exists $h\in \Mod(S)$ with $hV\perp V$, we have $hV \nest W_i$ for some $i$. Without loss of generality we may assume $hV\nest W_1$. Note that $S\setminus W_1$ is a connected surface $U_1$. We emphasize in the following two results that $S$ is a closed surface.

\begin{thm}
In the notation outlined above, suppose that $g(W_1)>g(U_1)$, and let $[\Mod(S)\curvearrowright X]\in\mathcal H(\Mod(S))$.  If there exists $\phi\in \Mod(S)$ supported on $V$ such that $\phi$ acts loxodromically on $X$, then $[\Mod(S)\curvearrowright X]$ is lineal and maximal.

\label{lineal}
\end{thm}

\begin{proof}
Consider the graph $\Gamma=\Gamma(V)$ with vertex set equal to the orbit $\Mod(S)\cdot V$ and edges joining pairs $hV$ and $kV$ whenever $hV\perp kV$. 

We first show that the graph $\Gamma$ is connected.  It suffices to show that for all elements $g\in \mathcal{G}$, for a fixed finite generating set $\mathcal{G}$ of $\Mod(S)$, there is a path from $V$ to $gV$ in $\Gamma$. We consider the Humphries generators $\mathcal{G}$, which are Dehn twists along the blue curves in Figure \ref{humphries}. Moreover, we suppose that $g(U_1)\geq 2$; the cases $g(U_1)\leq 1$ are handled in a nearly identical manner.

\begin{figure}[h]
\begin{tabular}{c c}
\def\svgwidth{225pt}
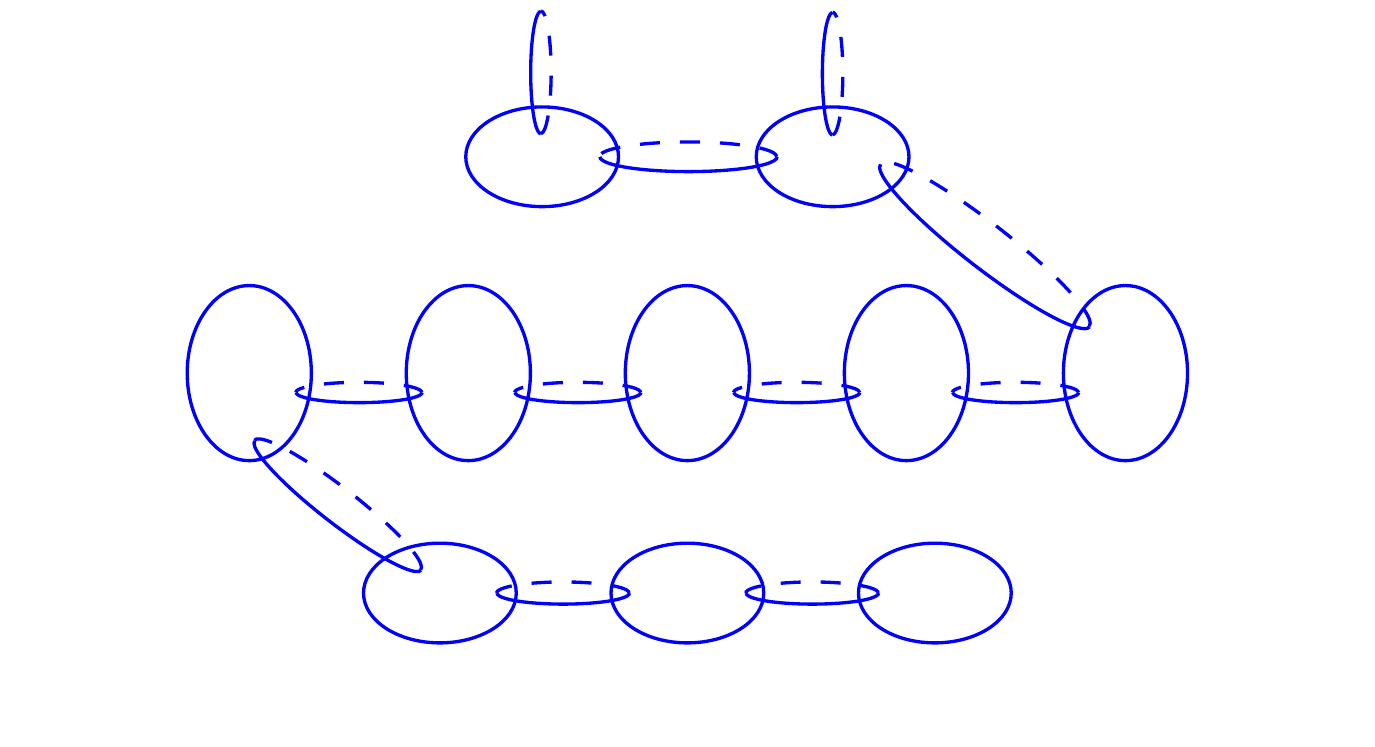 & 
\def\svgwidth{225pt}
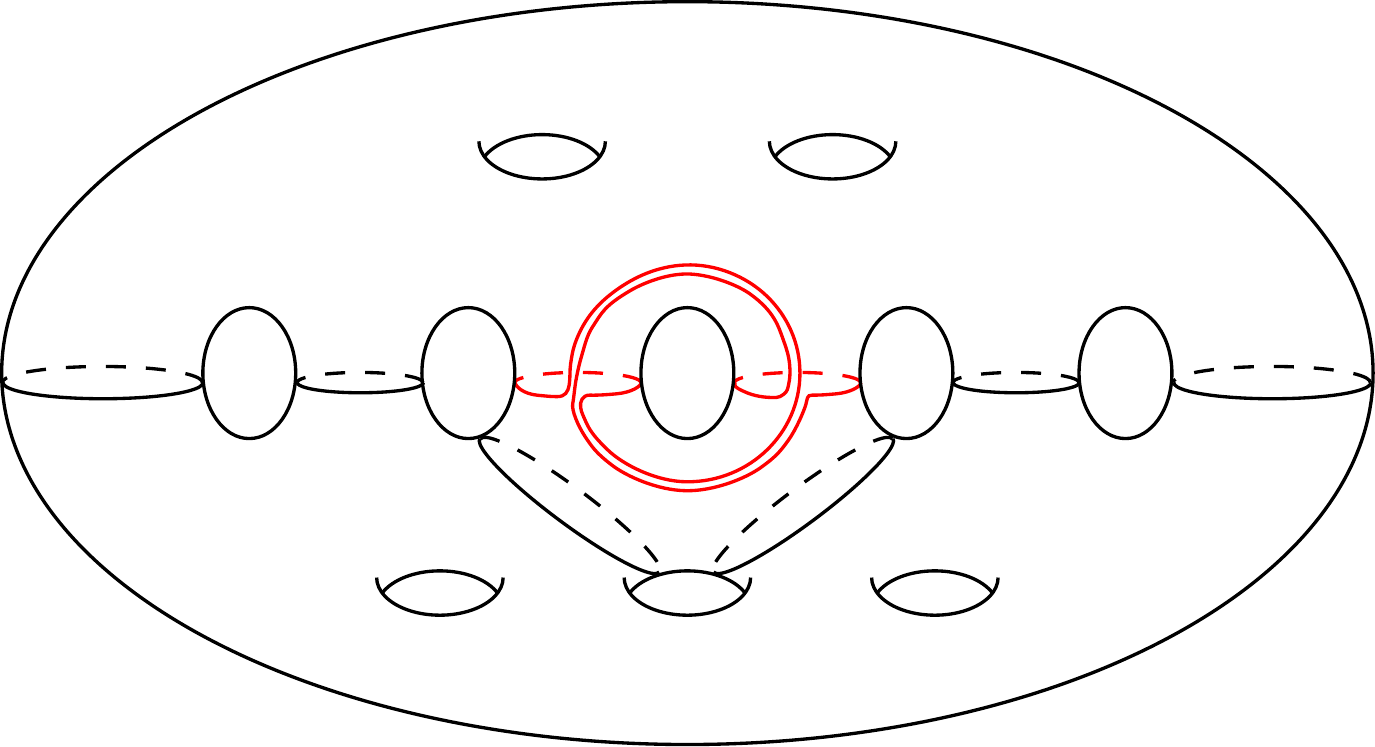 \\
\end{tabular}
\caption{Dehn twists on the blue curves on the left form the \textit{Humphries generating set} $\mathcal{G}$ of $\Mod(S)$. The figure on the right shows how $U_1$ transforms after applying the generator $g_i$. The surface $fU_1$ is bounded by the black curves and avoids the red curves, which are components of $\partial g_i U_1$.}
\label{humphries}
\end{figure}

We single out the Dehn twists $g_i$ around the middle curves as shown in Figure \ref{humphries}. For $g\in \mathcal{G} \setminus \{g_1,\ldots,g_n\}$ we have $gU_1=U_1$. Since $V\nest U_1$ and $gV\nest gU_1=U_1$ whereas $hV\nest W_1$, we have $V\perp hV$ and $hV\perp gV$. Thus, $V,hV,gV$ constitutes a path from $V$ to $gV$ in $\Gamma$. On the other hand, for a generator $g_i$ we have $g_iU_1\neq U_1$. In this case  there is an element $fU_1$ of the orbit of $U_1$ with $fU_1\perp U_1$ and $f U_1 \perp g_iU_1$. Hence we have $V\perp fV$ and $fV\perp g_iV$, and thus $V,fV,g_iV$ is a path in $\Gamma$.  Therefore, $\Gamma$ is connected.

We now show how this implies the theorem.  Let $\phi$ act loxodromically on $X$ with fixed points $\phi^{\pm}\in \partial X$. It will be convenient to take a power of $\phi$ to assume without loss of generality that $\phi$ fixes $\partial V$ pointwise (up to isotopy). For $h\in \Mod(S)$, the conjugate $h\phi h^{-1}$ is loxodromic with fixed points $h\phi^{\pm}$. Consider a path $V=h_0V, h_1V,\ldots, h_rV=hV$ in $\Gamma$. For $i$ between $0$ and $r$ we have:
\begin{itemize}
\item $h_iV\perp h_{i+1}V$, 
\item $h_i\phi h_i^{-1}$ is supported on $h_iV$ and fixes $\partial h_iV$ pointwise, and 
\item $h_{i+1}\phi h_{i+1}^{-1}$ is supported on $h_{i+1}V$ and fixes $\partial h_{i+1}V$ pointwise.
\end{itemize}
 Hence $h_i\phi h_i^{-1}$ and $h_{i+1}\phi h_{i+1}^{-1}$ commute and therefore fix the same pair of points on $\partial X$. That is, $h_i\phi^{\pm}=h_{i+1}\phi^{\pm}$. From this string of equalities, we find that $h\phi^{\pm}=\phi^{\pm}$. Thus all of $\Mod(S)$ fixes $\phi^{\pm}\in \partial X$, and since $\Mod(S)\curvearrowright X$ is cobounded, we must have that $X$ is a quasi-line. 
 
 To see that $[\Mod(S)\curvearrowright X]$ is maximal, suppose that there exists $[\Mod(S)\curvearrowright Y]\in \mc H(G)$ such that $[\Mod(S)\curvearrowright X]\preceq [\Mod(S)\curvearrowright Y]$. Then $\phi$ acts loxodromically on $Y$, and so the same argument shows that $Y$ is a quasi-line. Since all lineal structures are minimal (\cite[Corollary~4.12]{abo}), we must in fact have $[\Mod(S)\curvearrowright X]=[\Mod(S)\curvearrowright Y]$.

\end{proof}

The following lemma shows that the assumption that $g(W_1)>g(U_1)$ in Theorem \ref{lineal} is not too restrictive.  

\begin{lem}
In the notation outlined above, we have $g(W_1)\geq g(U_1)$.
\end{lem}

\begin{proof}
Since $hV\nest W_1$, some connected component of $S\setminus hV$ must contain $U_1$, and this component must be one of the subsurfaces $hW_1,\ldots,hW_n$. If $U_1\nest hW_1$, then  $g(U_1)\leq g(hW_1)=g(W_1)$, and the proof is complete. Otherwise  $U_1\nest hW_i$ for some $i>1$. Without loss of generality we suppose that $U_1\nest hW_2$. We will show in this case that in fact $n=2$ and $V$ is an annulus.

We consider the complexity $\xi$ of the subsurfaces involved. We note that $b(U_1)=b(W_1)$. The genus of $U_1$ is equal to $g(V)+\sum_{i\geq 2} g(W_i)+\sum_{i\geq 2} (b(W_i)-1)$ (here the last term comes from considering the contribution of the boundary components of the subsurfaces $W_i$ to the genus of $U_1$). Thus, we have \[\xi(U_1)=3g(V)+3\sum_{i=2}^n g(W_i) + 3\sum_{i=2}^n \left(b(W_i)-1\right)-3 +b(W_1).\] Once again, it is important to note that $S$ is a closed surface in order for the above calculation to make sense. On the other hand, $\xi(hW_2)=3g(W_2)-3+b(W_2)$. Since $U_1\nest hW_2$, we have \[3g(V)+3\sum_{i=2}^n g(W_i) + 3\sum_{i=2}^n \left(b(W_i)-1\right) -3+b(W_1)=\xi(U_1)\leq \xi(hW_2)=3g(W_2)-3+b(W_2).\] Subtracting $3g(W_2)-3$ from both sides yields \[3g(V)+3\sum_{i=3}^ng(W_i)+3\sum_{i=2}^n \left(b(W_i)-1\right)+b(W_1)\leq b(W_2).\] If $b(W_2)>1$, then $3(b(W_2)-1)>b(W_2)$, and we have a contradiction. Hence  $b(W_2)=1$, and the above inequality reduces to \[3g(V)+3\sum_{i=3}^ng(W_i)+3\sum_{i=3}^n(b(W_i)-1)+b(W_1)\leq 1.\] Since $b(W_1)\geq 1$ the only way for this inequality to hold is if:

\begin{itemize}
\item $g(V)=0$,
\item $g(W_i)=0$ for $i\geq 3$,
\item $b(W_i)=1$ for $i\geq 3$, and
\item $b(W_1)=1$.
\end{itemize}

Since none of the boundary components of $V$ are homotopically trivial, we cannot have $g(W_i)=0$ and $b(W_i)=1$ simultaneously. Hence we find that in fact $n=2$ and $V$ is an annulus with $S\setminus V=W_1\sqcup W_2$. Since $hV\nest W_1$ it is straightforward in this case to see that $g(W_1) \geq g(W_2)=g(U_1)$.
\end{proof}

%Recall that  is any mapping class that is supported on the proper connected subsurface $V\nest S$

We are now ready to prove Theorem \ref{thm:maxlineal}.  We recall its statement for the convenience of the reader.%We recall the statement of Theorem :

\maxlineal*

\begin{proof}
Let $\alpha$ be a nonseparating curve. By the previous section, there exists a cobounded action $\Mod(S)\curvearrowright X$ where $X$ is a quasi-line and $T_\alpha$ acts loxodromically. Moreover, $[\Mod(S)\curvearrowright X]\in \H(\Mod(S))$ is maximal by Theorem \ref{lineal}. Since $\Mod(S)\curvearrowright X$ is lineal and lineal hyperbolic actions are also minimal by \cite[Corollary~4.12]{abo}, this structure is comparable only to the equivalence class of the trivial action.
\end{proof}

\begin{rem}
Using quasimorphisms allows Theorem \ref{lineal} to be applied for many other mapping classes $\phi$. As an example, if $\phi$ is a chiral pseudo-Anosov supported on a subsurface $V\subset S$ with $g(V)<g(S)/2$ and with one boundary component, then there exists a homogeneous quasimorphism $q:\Mod(S)\to \R$ not vanishing on $\phi$ (see \cite{scl}). This quasimorphism then gives rise to a lineal action of $\Mod(S)$ in which $\phi$ acts loxodromically. By Theorem \ref{lineal}, this action is maximal.
\end{rem}

We now turn our attention to elements of $\mc H(\Mod(S))$ that are not lineal.  For the following theorem, we use $\Mod(V)$ to denote the elements of $\Mod(S)$ supported on $V$. This is technically larger than the mapping class group of $V$ because it includes elements which permute the components of $\partial V$. We denote by $\PMod(V)$ the elements of $\Mod(S)$ supported on $V$ that fix $\partial V$ pointwise. For the convenience of the reader, we recall the notation $V,U_1,W_1,\ldots,W_n$: $V$ is a subsurface which is disjoint from some surface in its $\Mod(S)$-orbit. The complementary subsurfaces to $V$ are the subsurfaces $W_1,\ldots,W_n$. Since $V$ is disjoint from a surface in its $\Mod(S)$-orbit there is some translate of $V$ contained in one of the subsurfaces $W_1,\ldots,W_n$. The numbering is chosen so that this translate is contained in $W_1$. Finally, $U_1$ is the complementary subsurface to $W_1$.

\begin{thm}
Suppose  that $g(W_1)>g(U_1)$. If $[\Mod(S)\curvearrowright X]\in\mathcal{H}(\Mod(S)) \setminus \mathcal{H}_\ell(\Mod(S))$, then either $[\Mod(S)\curvearrowright X]\in \mathcal{H}^{qp}(\Mod(S))$ or the action $\Mod(V)\curvearrowright X$ is elliptic.
\end{thm}

\begin{proof}
Since $X$ is not a quasi-line, Theorem \ref{lineal} implies that $\Mod(V)$ contains no loxodromics with respect to the action on $X$. By the classification of hyperbolic actions, $\Mod(V)\curvearrowright X$ is parabolic or elliptic.

Suppose that $\Mod(V)\curvearrowright X$ is parabolic so that $\Mod(V)$ fixes a single point $p\in \partial X$. We will show that all of $\Mod(S)$ fixes $p$.  By the classification of hyperbolic actions, since $\Mod(S)\curvearrowright X$ is cobounded, it must then be quasi-parabolic.

We first show that \begin{equation}\label{eqn:fixpmod}\Fix(\PMod(V))=\Fix(\Mod(V))=\{p\}.\end{equation} Since $\PMod(V)\leq \Mod(V)$, we clearly have $p\in \Fix(\PMod(V))$. However if $| \Fix(\PMod(V)) |>1$ then $\PMod(V)\curvearrowright X$ is elliptic, since $\PMod(V)\leq \Mod(V)$ contains no loxodromics. In this case, since $\Mod(V)$ contains an elliptic subgroup of finite index, it must be elliptic itself. But we are supposing that $\Mod(V)$ is parabolic, and thus (\ref{eqn:fixpmod}) follows.

From the proof of Theorem \ref{lineal}, we know that the graph $\Gamma=\Gamma(V)$ is connected. Consider $h\in \Mod(S)$ and a path $V=h_0 V, h_1V,\ldots,h_rV=hV$ in $\Gamma$. We claim that $\PMod(h_i V)$ fixes $p$ for all $i$. To prove the claim,
suppose for induction that $\PMod(h_i V)\cdot p=p$. We want to show that $\PMod(h_{i+1})\cdot p=p$ as well. We have that $h_iV\perp h_{i+1}V$;   every element of $\PMod(h_iV)$ fixes $\partial h_i V$ pointwise; and every element of $\PMod(h_{i+1}V)$ fixes $\partial h_{i+1}V$ pointwise. Thus every element of $\PMod(h_i V)$ commutes with every element of $\PMod(h_{i+1}V)$. In particular, if $\phi \in \PMod(h_{i+1}V)$ then it fixes setwise the set $\Fix(\psi)$ for every $\psi\in \PMod(h_i V)$.   Therefore $\phi$ also fixes setwise the set \[\bigcap_{\psi \in \PMod(h_i V)} \Fix(\psi)=\Fix(\PMod(h_i V))=\{p\}.\] Since $\phi\in \PMod(h_{i+1}V)$ is arbitrary, this proves the claim.

The theorem now follows because  \[\{p\}=\Fix(\PMod(h_i V))=\Fix(h_i \PMod(V) h_i^{-1})=h_i\Fix(\PMod(V))=\{h_ip\},\] and in particular $hp=h_np=p$.
\end{proof}

%\begin{ques}
%\begin{enumerate}
%\item If $\phi$ is reducible and supported on a subsurface $V$ such that $\{hV\mid h\in G\}$ is pairwise transverse, then there exists $[X]\in\mathcal H(G)$ such that $\Gamma(G,X)$ is a quasi-tree of curve graphs.  Is this action maximal in $\mathcal H(G)$?
%\item In the previous theorem, if genus of $W_1$ is equal to the genus of $V_1$, and $\phi$ does not restrict to a Dehn twist on any boundary component of $W_1$, if $\phi\in \mathcal{L}([X])$ then is $X$ a quasi-line?
%\item Does $G$ admit a quasi-parabolic structure?
%\end{enumerate}
%\end{ques}

\section{Fundamental groups of flip graph manifolds}
\label{sec:flipgraph}

In this section, we prove that the fundamental groups of most flip graph manifolds do not have largest hyperbolic actions by applying Lemma \ref{lem:mainlem}. As for mapping class groups in the previous section, we will construct two quasi-trees of metric spaces. However, we need to divide this construction into two cases, depending on the flip graph manifold (in particular, the structure of its underlying graph).  In the first case (Section \ref{sec:noloops}), we use the action of the fundamental group of the flip graph manifold on the quasi-trees of metric spaces to directly apply Lemma \ref{lem:mainlem} and conclude.  In the second case (Section \ref{sec:loops}), we are only able to obtain an action of a finite-index subgroup of the fundamental group on the quasi-trees of metric spaces, which is not sufficient to apply our main lemma.  In this case, we will use the quasi-trees of metric spaces to construct quasimorphisms, which will in turn allow us to construct two lineal actions to which we can apply our main lemma.

\subsection{Flip graph manifolds}
We first recall the definition and some fundamental facts about flip graph manifolds.
A connected 3--manifold $M$ is a \textit{flip graph manifold} if it has the following form. The manifold $M$ is made up of finitely many \textit{pieces} which are trivial circle bundles $S\times S^1$ where $S$ is a surface with negative Euler characteristic and with boundary (and no punctures). A boundary component $c$ of the base $S$ of a piece defines a torus boundary component $c \times S^1$ of the piece, and these torus boundary components are glued in pairs by orientation-reversing homeomorphisms which interchange the boundary component and fiber directions of two distinct torus boundary components.

The manifold $M$ is  homeomorphic to a graph of spaces where the vertex spaces are the pieces of the decomposition and edge spaces correspond to boundary tori. We denote by $\Gamma$ the underlying graph. The universal cover $\tilde{M}$ is homeomorphic to a tree of spaces with an underlying tree which we denote $\tilde{\Gamma}$. In $\tilde{M}$ the vertex spaces are universal covers of the pieces $S\times S^1$, which are homeomorphic to products of closed convex subsets of the hyperbolic plane $\hyp^2$ with $\R$. These vertex spaces have boundary consisting of infinitely many copies of the plane $\R^2$ and these correspond to the edge spaces of $\tilde{M}$. For simplicity, we will refer to the vertex spaces of $\tilde{M}$ as \textit{lifts} of the pieces of $M$.

We endow $M$ with a locally CAT(0) metric as follows. For a piece $S\times S^1$, the base $S$ admits a hyperbolic metric with geodesic boundary components of length one. Further, we endow $S^1$ with a Euclidean metric of length one by identifying it with the unit interval $[0,1]$ with the endpoints identified. We endow the piece $S\times S^1$ with the product $\ell^2$ metric. We also require the identifications of torus boundary components to be given by orientation-reversing isometries that have the form $(x,y)\mapsto (y,x)$ in appropriate coordinates.

The universal cover $\tilde{M}$ inherits a pullback metric. If $X=S\times S^1$, then its lifts are each isometric to $\tilde{X}=\tilde{S}\times \R$, where $\tilde{S}$ has the pullback metric induced by the chosen hyperbolic metric on $S$ and $\R$ has the standard Euclidean metric. The universal cover $\tilde{S}$ is isometric to a closed convex subset of $\hyp^2$ with infinitely many geodesic boundary components. The vertex spaces $\tilde{S}\times \R$ are glued together along copies of the Euclidean plane $\R^2$ where the identifications are given by orientation reversing isometries $(x,y)\mapsto (y,x)$ in appropriate coordinates.

For a piece $X$ and a lift $\tilde{X}$ to $\tilde{M}$, isometric to $\tilde{S}\times \R$, the relation $(x,t)\sim_{\tilde{X}} (y,t)$ for $x,y\in \tilde{S}$ gives rise to a quotient space $\ell_{\tilde{X}}=\tilde{X}/\sim_{\tilde{X}}$ that inherits a metric with respect to which it is isometric to the real line. We denote by $p_{\tilde{X}}:\tilde{X}\to \ell_{\tilde{X}}$ the Lipschitz quotient map.

The fundamental group $\pi_1(M)$ acts by isometries on $\tilde{M}$. If $X$ is a piece of $M$, $\tilde{X}$ is a lift, and $g\in \pi_1(M)$, then $g\tilde{X}$ is another lift and the isometry \[g|_{\tilde{X}}:\tilde{X}\to g\tilde{X}\] respects the equivalence relations on $\tilde{X}$ and $g\tilde{X}$. In other words, if $p,q\in \tilde{X}$ and $p\sim_{\tilde{X}} q$ then $gp\sim_{g\tilde{X}} gq$. Hence $g$ induces a map $\ell_{\tilde{X}}\to \ell_{g\tilde{X}}$. This map is an isometry.

\subsection{Projections}

Let $v$ be a vertex of $\tilde{\Gamma}$. Then the vertex space $\tilde{M}_v$ is bounded by infinitely many Euclidean planes. If $P$ and $Q$ are two distinct such planes, we may consider the set of points of $Q$ which are closest to $P$. Denote this set by $\rho_Q(P)$. In other words, we define \[\rho_Q(P)=\{q\in Q: d(q,P)\leq d(q',P) \text{ for any } q'\in Q\}.\] Then $\rho_Q(P)$ is a geodesic line in $Q$. Parametrizing $\tilde{M}_v$ as $H_v\times \R$, where $H_v$ is a closed convex subset of $\hyp^2$, $P$ and $Q$ have the form $\alpha \times \R$ and $\beta \times \R$, respectively, where $\alpha$ and $\beta$ are boundary components of $H_v$. If $a$ is the closest point on $\beta$ to $\alpha$, then $\rho_Q(P)$ is parametrized as $\{a\}\times \R$. See Figure \ref{projection}.

\begin{figure}[h]

\centering
\begin{tabular}{c c}

\centering
\def\svgwidth{0.4\textwidth}
%% Accompanies image file 'hypplane.pdf' (pdf, eps, ps)
\begingroup%
  \makeatletter%
  \providecommand\color[2][]{%
    \errmessage{(Inkscape) Color is used for the text in Inkscape, but the package 'color.sty' is not loaded}%
    \renewcommand\color[2][]{}%
  }%
  \providecommand\transparent[1]{%
    \errmessage{(Inkscape) Transparency is used (non-zero) for the text in Inkscape, but the package 'transparent.sty' is not loaded}%
    \renewcommand\transparent[1]{}%
  }%
  \providecommand\rotatebox[2]{#2}%
  \newcommand*\fsize{\dimexpr\f@size pt\relax}%
  \newcommand*\lineheight[1]{\fontsize{\fsize}{#1\fsize}\selectfont}%
  \ifx\svgwidth\undefined%
    \setlength{\unitlength}{493.6083696bp}%
    \ifx\svgscale\undefined%
      \relax%
    \else%
      \setlength{\unitlength}{\unitlength * \real{\svgscale}}%
    \fi%
  \else%
    \setlength{\unitlength}{\svgwidth}%
  \fi%
  \global\let\svgwidth\undefined%
  \global\let\svgscale\undefined%
  \makeatother%
  \begin{picture}(1,1)%
    \lineheight{1}%
    \setlength\tabcolsep{0pt}%
    \put(0,0){\includegraphics[width=\unitlength,page=1]{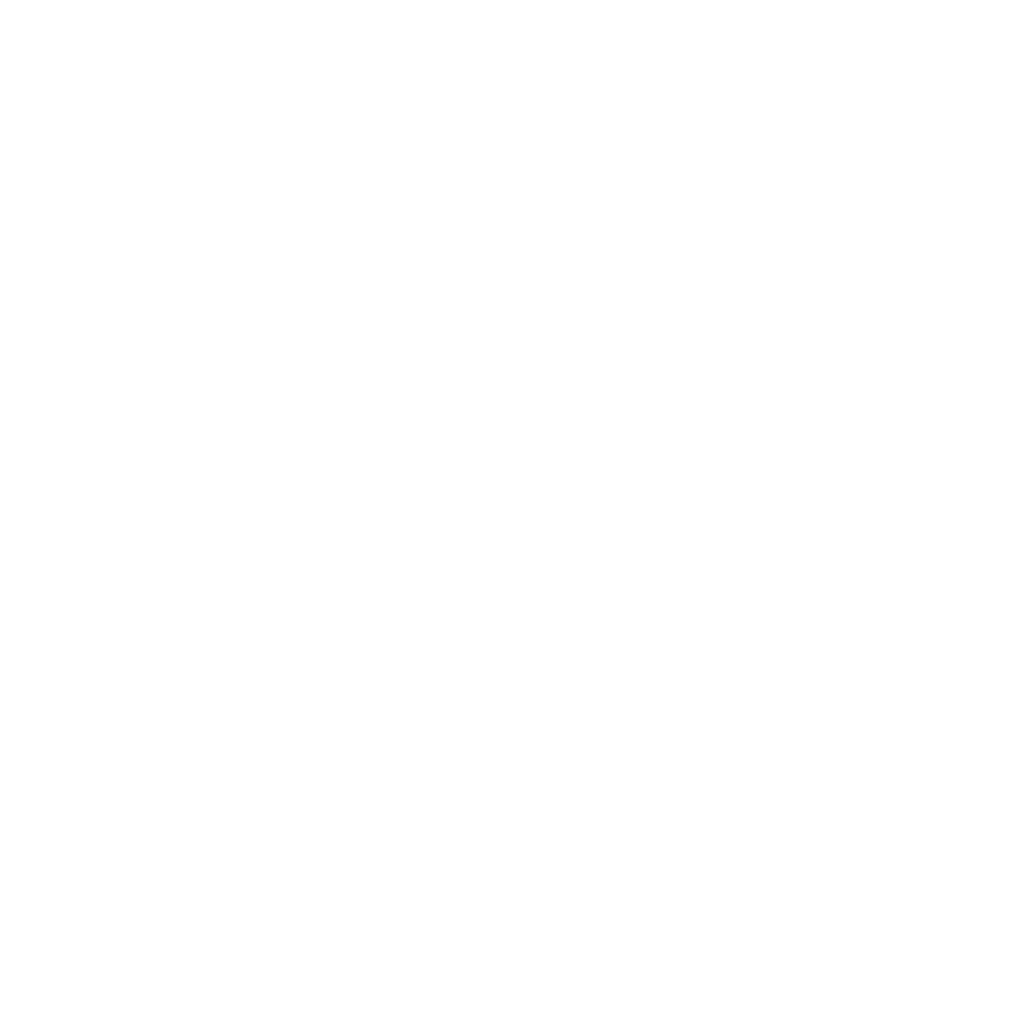}}%
    \put(0.55693982,0.50565411){\color[rgb]{0,0,0}\makebox(0,0)[lt]{\lineheight{1.25}\smash{\begin{tabular}[t]{l}$H_v$\end{tabular}}}}%
    \put(0.15728069,0.48348528){\color[rgb]{0,0,0}\makebox(0,0)[lt]{\lineheight{1.25}\smash{\begin{tabular}[t]{l}$\alpha$\end{tabular}}}}%
    \put(0.66321696,0.25566537){\color[rgb]{0,0,0}\makebox(0,0)[lt]{\lineheight{1.25}\smash{\begin{tabular}[t]{l}$\beta$\end{tabular}}}}%
    \put(0,0){\includegraphics[width=\unitlength,page=2]{hypplane.pdf}}%
  \end{picture}%
\endgroup%
 &

\centering
\def\svgwidth{0.4\textwidth}
%% Accompanies image file 'projection.pdf' (pdf, eps, ps)
\begingroup%
  \makeatletter%
  \providecommand\color[2][]{%
    \errmessage{(Inkscape) Color is used for the text in Inkscape, but the package 'color.sty' is not loaded}%
    \renewcommand\color[2][]{}%
  }%
  \providecommand\transparent[1]{%
    \errmessage{(Inkscape) Transparency is used (non-zero) for the text in Inkscape, but the package 'transparent.sty' is not loaded}%
    \renewcommand\transparent[1]{}%
  }%
  \providecommand\rotatebox[2]{#2}%
  \newcommand*\fsize{\dimexpr\f@size pt\relax}%
  \newcommand*\lineheight[1]{\fontsize{\fsize}{#1\fsize}\selectfont}%
  \ifx\svgwidth\undefined%
    \setlength{\unitlength}{493.6083696bp}%
    \ifx\svgscale\undefined%
      \relax%
    \else%
      \setlength{\unitlength}{\unitlength * \real{\svgscale}}%
    \fi%
  \else%
    \setlength{\unitlength}{\svgwidth}%
  \fi%
  \global\let\svgwidth\undefined%
  \global\let\svgscale\undefined%
  \makeatother%
  \begin{picture}(1,1.61027072)%
    \lineheight{1}%
    \setlength\tabcolsep{0pt}%
    \put(0,0){\includegraphics[width=\unitlength,page=1]{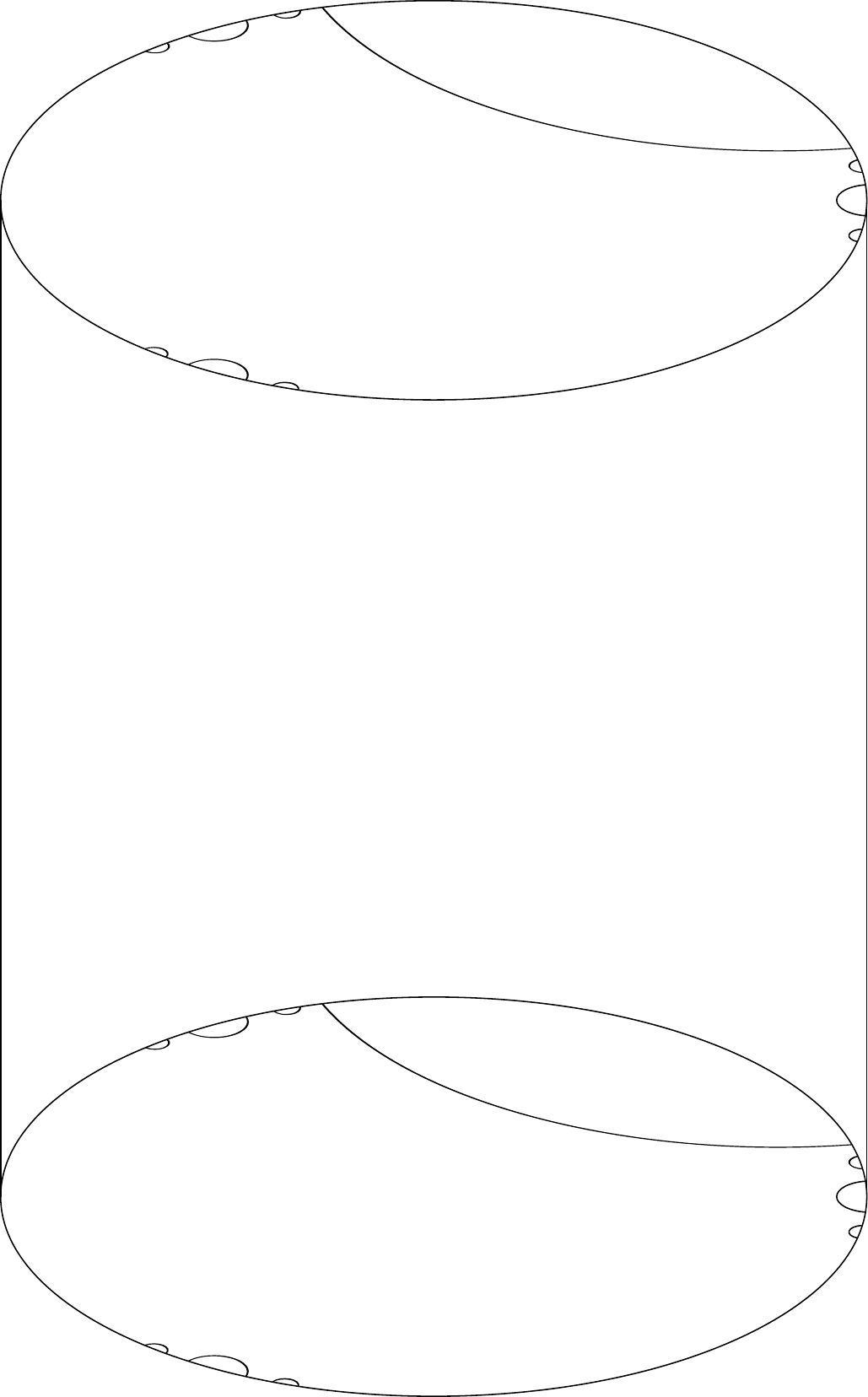}}%
    \put(0.26914236,1.43893413){\color[rgb]{0,0,0}\makebox(0,0)[lt]{\lineheight{1.25}\smash{\begin{tabular}[t]{l}$P$\end{tabular}}}}%
    \put(0.70819012,1.34420803){\color[rgb]{0,0,0}\makebox(0,0)[lt]{\lineheight{1.25}\smash{\begin{tabular}[t]{l}$Q$\end{tabular}}}}%
    \put(0.52256414,0.75501465){\color[rgb]{0,0,0}\makebox(0,0)[lt]{\lineheight{1.25}\smash{\begin{tabular}[t]{l}$\rho_Q(P)$\end{tabular}}}}%
    \put(0,0){\includegraphics[width=\unitlength,page=2]{projection.pdf}}%
  \end{picture}%
\endgroup%

\end{tabular}

\caption{The definition of the projection of one boundary plane onto another in a vertex space of $\tilde{M}$. The dotted line denotes a shortest geodesic from $\alpha$ to $\beta$ (which is orthogonal to $\alpha$ and $\beta$ at its endpoints).}
\label{projection}
\end{figure}

If $P$ and $Q$ are distinct boundary planes of $\tilde{M}_v$, as above, $w$ is adjacent to $v$ in $\tilde{\Gamma}$,  and $\tilde{M}_w$ is glued to $\tilde{M}_v$ along $Q$, then the image $p_{\tilde{M}_w}(\rho_Q(P))$ is a single point of $\ell_{\tilde{M}_w}$. If $v$ and $w$ are vertices of $\tilde{\Gamma}$ at distance at least two apart then we define a projection from $\ell_{\tilde{M}_v}$ onto a point of $\ell_{\tilde{M}_w}$ as follows. Consider the unique geodesic $[v,w]$ oriented from $v$ to $w$ in $\tilde{\Gamma}$. Let $u',u,w$ be the last three vertices of $[v,w]$, occurring in that order. Then $\tilde{M}_u$ is glued to $\tilde{M}_{u'}$ along a unique Euclidean boundary plane $P$ and to $\tilde{M}_w$ along a Euclidean boundary plane $Q$ which is distinct from $P$. We define the projection of $\ell_{\tilde{M}_v}$ to $\ell_{\tilde{M}_w}$ to be the point $p_{\tilde{M}_w}(\rho_Q(P))$. We denote this point by $\pi_{\tilde{M}_w}(\tilde{M}_v)$. See Figure \ref{treegeod}.

\begin{figure}[h]
\centering

\def\svgwidth{0.8\textwidth}
%% Accompanies image file 'treegeod.pdf' (pdf, eps, ps)
\begingroup%
  \makeatletter%
  \providecommand\color[2][]{%
    \errmessage{(Inkscape) Color is used for the text in Inkscape, but the package 'color.sty' is not loaded}%
    \renewcommand\color[2][]{}%
  }%
  \providecommand\transparent[1]{%
    \errmessage{(Inkscape) Transparency is used (non-zero) for the text in Inkscape, but the package 'transparent.sty' is not loaded}%
    \renewcommand\transparent[1]{}%
  }%
  \providecommand\rotatebox[2]{#2}%
  \newcommand*\fsize{\dimexpr\f@size pt\relax}%
  \newcommand*\lineheight[1]{\fontsize{\fsize}{#1\fsize}\selectfont}%
  \ifx\svgwidth\undefined%
    \setlength{\unitlength}{694.10453256bp}%
    \ifx\svgscale\undefined%
      \relax%
    \else%
      \setlength{\unitlength}{\unitlength * \real{\svgscale}}%
    \fi%
  \else%
    \setlength{\unitlength}{\svgwidth}%
  \fi%
  \global\let\svgwidth\undefined%
  \global\let\svgscale\undefined%
  \makeatother%
  \begin{picture}(1,0.10571034)%
    \lineheight{1}%
    \setlength\tabcolsep{0pt}%
    \put(0,0){\includegraphics[width=\unitlength,page=1]{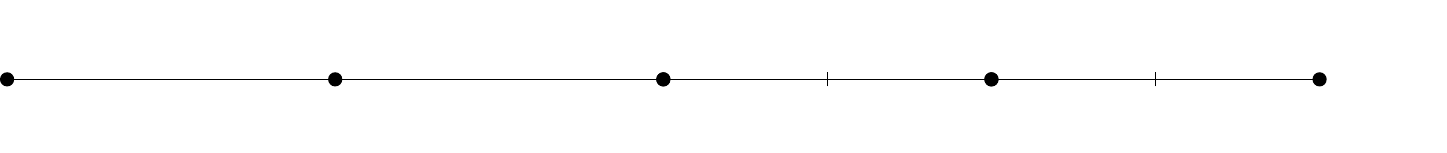}}%
    \put(0.00498484,0.00761858){\color[rgb]{0,0,0}\makebox(0,0)[lt]{\lineheight{1.25}\smash{\begin{tabular}[t]{l}$v$\end{tabular}}}}%
    \put(0.91262911,0.00761858){\color[rgb]{0,0,0}\makebox(0,0)[lt]{\lineheight{1.25}\smash{\begin{tabular}[t]{l}$w$\end{tabular}}}}%
    \put(0.68571804,0.00761858){\color[rgb]{0,0,0}\makebox(0,0)[lt]{\lineheight{1.25}\smash{\begin{tabular}[t]{l}$u$\end{tabular}}}}%
    \put(0.56685981,0.07245031){\color[rgb]{0,0,0}\makebox(0,0)[lt]{\lineheight{1.25}\smash{\begin{tabular}[t]{l}$P$\end{tabular}}}}%
    \put(0.79377086,0.07245031){\color[rgb]{0,0,0}\makebox(0,0)[lt]{\lineheight{1.25}\smash{\begin{tabular}[t]{l}$Q$\end{tabular}}}}%
    \put(0.45880685,0.00761851){\color[rgb]{0,0,0}\makebox(0,0)[lt]{\lineheight{1.25}\smash{\begin{tabular}[t]{l}$u'$\end{tabular}}}}%
  \end{picture}%
\endgroup%

\caption{The projection of $\ell_{\tilde{M}_v}$ to $\ell_{\tilde{M}_w}$ is obtained by projecting the boundary plane $P$ onto the plane $Q$ and then projecting the resulting line to the vertical direction of $\tilde{M}_w$.}
\label{treegeod}
\end{figure}

\noindent We remark that the projection $\pi_{\tilde{M}_w}(\tilde{M}_v)$ coarsely agrees with the composition of the \emph{closest point projection} of $\tilde{M}_v$ to $\tilde{M}_w$, together with the quotient $\tilde{M}_w\to \ell_{\tilde{M}_w}$. However, we do not use this fact in the sequel and leave the proof to the interested reader.

\subsection{The case that $\Gamma$ contains no loops}
\label{sec:noloops}

In this subsection we assume that $\Gamma$ contains no loops. In the next section we explain how this restriction may be removed.

 Let $X$ and $Y$ be two adjacent vertex spaces of $M$. In $\tilde{M}$ we may choose lifts $\tilde{X}_0$ and $\tilde{Y}_0$ which are glued along a common Euclidean plane $P$. Writing $\tilde{X}_0=H_0\times \R$ and $\tilde{Y}_0=H_0'\times \R$ where $H_0$ and $H_0'$ are closed convex subsets of $\hyp^2$, there are boundary components $\alpha_0$ and $\beta_0$ of $H_0$ and $H_0'$, respectively, such that $P$ is identified with the product $\alpha_0\times \beta_0$. Moreover, there are elements $a$ and $b$ of $\pi_1(M)$ (corresponding to orthogonal simple closed geodesics in the boundary torus along which $X$ and $Y$ are glued) such that

\begin{itemize}
\item $a$ acts on $\tilde{X}_0$ as $\phi\times \operatorname{id}$ in the product structure $H_0\times \R$ where $\phi$ is a loxodromic isometry of $H_0$ with axis $\alpha_0$,
\item $b$ acts on $\tilde{Y}_0$ as $\psi \times \operatorname{id}$ in the product structure $H_0'\times \R$ where $\psi$ is a loxodromic isometry of $H_0'$ with axis $\beta_0$.
\end{itemize}

\noindent Consequently we see that $a$ and $b$ commute and fix both of the domains $\tilde{X}_0$ and $\tilde{Y}_0$ setwise. In this section we prove:

\begin{thm}
\label{noloops}
There exists a hyperbolic space $\mathcal{C}(\mathbf{X})$ with an action of $\pi_1(M)$ by isometries such that $a$ acts elliptically on $\mathcal{C}(\mathbf{X})$ while $b$ acts loxodromically. Similarly there exists a space $\mathcal{C}(\mathbf{Y})$ with an action of $\pi_1(M)$ in which $a$ acts loxodromically and $b$ acts elliptically.
\end{thm}

The spaces $\mathcal{C}(\mathbf{X})$ and $\mathcal{C}(\mathbf{Y})$ are quasi-trees of metric spaces as described in Section \ref{sec:bbf}. The constructions are completely analogous, so we focus only on the case of $\mathcal{C}(\mathbf{X})$.

The set of domains $\mathbf{X}$ is the set of lifts of $X$ to $\tilde{M}$. In particular, our chosen lift $\tilde{X}_0$ is an element of  $\mathbf{X}$. Associated to a domain $A\in \mathbf{X}$, there is an associated hyperbolic space $\mathcal{C}(A)=\ell_A$. We define the projections $\pi_B(A)$ for $A,B\in \mathbf{X}$ as above. These are well-defined because the vertices in $\tilde{\Gamma}$ corresponding to $A$ and $B$ (i.e., the vertices $v$ and $w$ such that $A=\tilde{M}_v$ and $B=\tilde{M}_w$) are distance at least two apart in $\tilde{\Gamma}$. This follows from the fact that $\Gamma$ has no loops, so that all edges of $\tilde{\Gamma}$ join vertices which project to distinct vertices in $\Gamma$ after quotienting by the action of $\pi_1(M)$. We define the distances \[d^\pi_C(A,B)=d_{\ell_C}(\pi_{C}(A),\pi_{C}(B))\] where $C\in \mathbf{X}$, $A,B\in \mathbf{X}\setminus \{C\}$, and $d_{\ell_C}$ denotes distance in the line $\ell_C$. The main technical result of this subsection is that these distances satisfy axioms (P0)-(P2) from Section \ref{sec:bbf}.

\begin{lem}
\label{axioms}
There exists $\theta>0$ large enough that the domains $\mathbf{X}$, spaces $\mathcal{C}(A)$ for $A\in \mathbf{X}$, and projections $\pi_{A}$ satisfy the axioms (P0)-(P2).
\end{lem}

Before proving the lemma, we  show how it can be used to prove Theorem \ref{noloops}.

\begin{proof}[Proof of Theorem \ref{noloops} using Lemma \ref{axioms}]
For $K$ large enough, the complex $\mathcal{C}(\mathbf{X})=\mathcal{C}_K(\mathbf{X})$ is a quasi-tree by Theorem \ref{thm:bbfqtree} and the lines $\C(A)$ are isometrically embedded in $\C_K(\mathbf{X})$ by Theorem \ref{thm:isomembed}.

The group $\pi_1(M)$ acts on the set $\mathbf{X}$ and permutes the associated lines $\mathcal{C}(A)$ by isometries. Moreover, it is easy to see that $\pi_1(M)$ preserves the projections $\pi_A$ (and hence also the distance functions $d_A^\pi$). Hence we obtain an action of $\pi_1(M)$ on $\mathcal{C}(\mathbf{X})$ by isometries. The elements $a$ and $b$ both fix $\mathcal{C}(\tilde{X}_0)$. The element $a$ fixes it pointwise whereas $b$ acts on it by translation. Since $\mathcal{C}(\tilde{X}_0)$ is isometrically embedded, this proves that $a$ is elliptic and $b$ is loxodromic, as desired.

By reversing the roles of $X$ and $Y$, we obtain a complex $\C(\mathbf{Y})$ on which $a$ is loxodromic and $b$ is elliptic.
\end{proof}

It now remains only to prove Lemma \ref{axioms}.
\begin{proof}[Proof of Lemma \ref{axioms}]
We will choose $\theta$ during the course of the proof.

Since $\pi_B(A)$ is a single point if $A,B\in \mathbf{X}$ are distinct, (P0) is trivially satisfied for any $\theta>0$.

We check (P2) first and then (P1). Consider two distinct domains $A,B\in \mathbf{X}$ corresponding to vertices $v$ and $w$, respectively, in the underlying tree $\tilde{\Gamma}$. Consider a third domain $C\in \mathbf{X}$ corresponding to a vertex $u$ of $\tilde{\Gamma}$. Let $[v,w]$ be the geodesic from $v$ to $w$ in $\tilde{\Gamma}$. Let $[u,u']$ be the unique geodesic from $u$ to $[v,w]$, where $u'$ is a vertex of $[v,w]$. If $d(u,u')\geq 2$ then we see immediately that $\pi_C(A)=\pi_C(B)$ so $d^\pi_C(A,B)=0$. Otherwise there are two cases:
\begin{enumerate}[(i)]
\item $u$ is a vertex of $[v,w]$, or
\item $u$ is joined by an edge to a vertex $u'$ of $[v,w]$.
\end{enumerate}

\begin{figure}[h]
\centering

\begin{tabular}{c c}
\centering
\def\svgwidth{0.4\textwidth}
%% Accompanies image file 'largedist1.pdf' (pdf, eps, ps)
\begingroup%
  \makeatletter%
  \providecommand\color[2][]{%
    \errmessage{(Inkscape) Color is used for the text in Inkscape, but the package 'color.sty' is not loaded}%
    \renewcommand\color[2][]{}%
  }%
  \providecommand\transparent[1]{%
    \errmessage{(Inkscape) Transparency is used (non-zero) for the text in Inkscape, but the package 'transparent.sty' is not loaded}%
    \renewcommand\transparent[1]{}%
  }%
  \providecommand\rotatebox[2]{#2}%
  \newcommand*\fsize{\dimexpr\f@size pt\relax}%
  \newcommand*\lineheight[1]{\fontsize{\fsize}{#1\fsize}\selectfont}%
  \ifx\svgwidth\undefined%
    \setlength{\unitlength}{694.10453256bp}%
    \ifx\svgscale\undefined%
      \relax%
    \else%
      \setlength{\unitlength}{\unitlength * \real{\svgscale}}%
    \fi%
  \else%
    \setlength{\unitlength}{\svgwidth}%
  \fi%
  \global\let\svgwidth\undefined%
  \global\let\svgscale\undefined%
  \makeatother%
  \begin{picture}(1,0.05574594)%
    \lineheight{1}%
    \setlength\tabcolsep{0pt}%
    \put(0,0){\includegraphics[width=\unitlength,page=1]{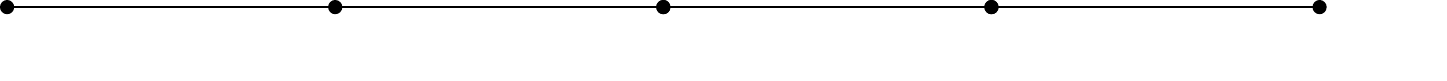}}%
    \put(0.00498484,0.00761865){\color[rgb]{0,0,0}\makebox(0,0)[lt]{\lineheight{1.25}\smash{\begin{tabular}[t]{l}$v$\end{tabular}}}}%
    \put(0.9126291,0.00761865){\color[rgb]{0,0,0}\makebox(0,0)[lt]{\lineheight{1.25}\smash{\begin{tabular}[t]{l}$w$\end{tabular}}}}%
    \put(0.45880678,0.00761859){\color[rgb]{0,0,0}\makebox(0,0)[lt]{\lineheight{1.25}\smash{\begin{tabular}[t]{l}$u$\end{tabular}}}}%
  \end{picture}%
\endgroup%
 &

\centering
\def\svgwidth{0.4\textwidth}
%% Accompanies image file 'largedist2.pdf' (pdf, eps, ps)
\begingroup%
  \makeatletter%
  \providecommand\color[2][]{%
    \errmessage{(Inkscape) Color is used for the text in Inkscape, but the package 'color.sty' is not loaded}%
    \renewcommand\color[2][]{}%
  }%
  \providecommand\transparent[1]{%
    \errmessage{(Inkscape) Transparency is used (non-zero) for the text in Inkscape, but the package 'transparent.sty' is not loaded}%
    \renewcommand\transparent[1]{}%
  }%
  \providecommand\rotatebox[2]{#2}%
  \newcommand*\fsize{\dimexpr\f@size pt\relax}%
  \newcommand*\lineheight[1]{\fontsize{\fsize}{#1\fsize}\selectfont}%
  \ifx\svgwidth\undefined%
    \setlength{\unitlength}{694.10453256bp}%
    \ifx\svgscale\undefined%
      \relax%
    \else%
      \setlength{\unitlength}{\unitlength * \real{\svgscale}}%
    \fi%
  \else%
    \setlength{\unitlength}{\svgwidth}%
  \fi%
  \global\let\svgwidth\undefined%
  \global\let\svgscale\undefined%
  \makeatother%
  \begin{picture}(1,0.30020545)%
    \lineheight{1}%
    \setlength\tabcolsep{0pt}%
    \put(0,0){\includegraphics[width=\unitlength,page=1]{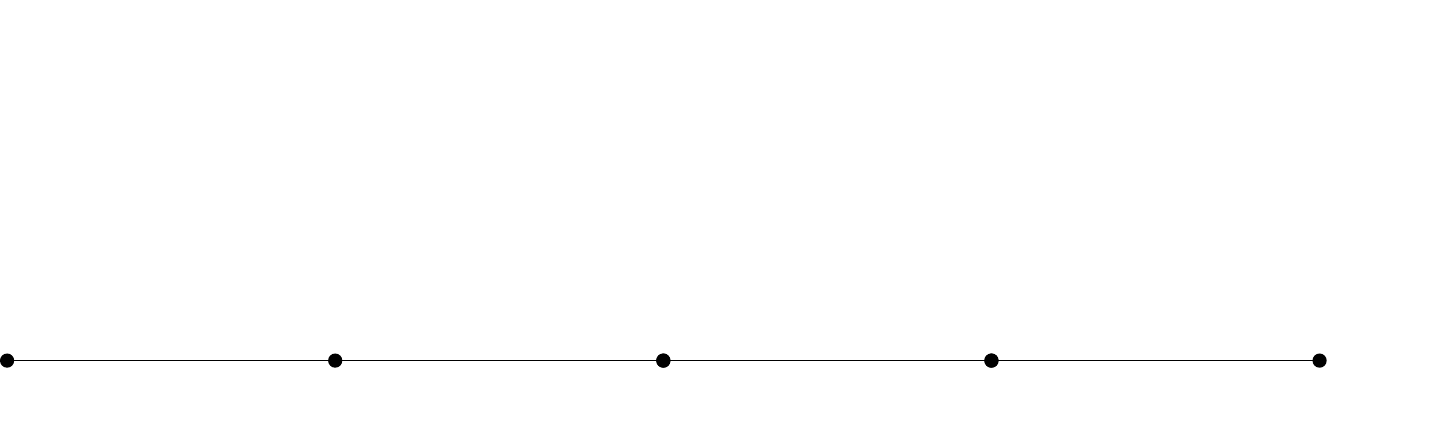}}%
    \put(0.00498484,0.00761858){\color[rgb]{0,0,0}\makebox(0,0)[lt]{\lineheight{1.25}\smash{\begin{tabular}[t]{l}$v$\end{tabular}}}}%
    \put(0.9126291,0.00761858){\color[rgb]{0,0,0}\makebox(0,0)[lt]{\lineheight{1.25}\smash{\begin{tabular}[t]{l}$w$\end{tabular}}}}%
    \put(0,0){\includegraphics[width=\unitlength,page=2]{largedist2.pdf}}%
    \put(0.48041744,0.26694542){\color[rgb]{0,0,0}\makebox(0,0)[lt]{\lineheight{1.25}\smash{\begin{tabular}[t]{l}$u$\end{tabular}}}}%
    \put(0.45880678,0.00761846){\color[rgb]{0,0,0}\makebox(0,0)[lt]{\lineheight{1.25}\smash{\begin{tabular}[t]{l}$u'$\end{tabular}}}}%
  \end{picture}%
\endgroup%

\end{tabular}

\caption{The two possible cases of domains $C$ (corresponding to the vertex $u$) with a large projection distance between $A$ and $B$ (corresponding to $v$ and $w$, respectively).}
\end{figure}

Clearly there are only finitely many domains $C\in \mathbf{X}$ corresponding to vertices of type (i). We claim that if $\theta$ is large enough then there are also finitely many domains $C$ with $d^\pi_C(A,B)>\theta$ corresponding to vertices of type (ii). Choose $\epsilon$ large enough that geodesic hexagons in $\hyp^2$ are $\epsilon$--thin. Also choose  $R>0$ small enough such that no two boundary components of the base $\tilde{S}$ of any vertex space $\tilde{S}\times \R$ of $\tilde{M}$ are $R$--close. Given any number $r>0$ there exists a number $\eta(r)$ such that if two geodesics in $\hyp^2$ are $2\epsilon$--close along segments of length greater than $\eta(r)$ then they are in fact $r$--close at some points. Hence, we see that no two boundary components of $\tilde{S}$ are $2\epsilon$--close along segments of length greater than $\eta(R)$. Set $\eta=\eta(R)$ and $\theta=6\epsilon+2\eta$. We claim that there are finitely many domains $C$ corresponding to vertices $u$ of type (ii) with $d^\pi_C(A,B)>\theta$.

If $C$ corresponds to a vertex $u$ of type (ii) (in other words $C=\tilde{M}_u$), the vertex space $\tilde{M}_{u'}$ of $\tilde{M}$ contains three boundary planes $P,Q,$ and $R$ such that \[\pi_C(A)=p_{\tilde{M}_u}(\rho_R(P)) \text{ and } \pi_C(B)=p_{\tilde{M}_u}(\rho_R(Q))\] (where $C=\tilde{M}_u$ is glued to $\tilde{M}_{u'}$ along $R$). Parametrizing $\tilde{M}_{u'}$ as $H\times \R$ with $H$ a closed convex subset of $\hyp^2$, there are boundary components $\alpha,\beta,$ and $\gamma$ of $H$ with $P=\alpha\times \R$, $Q=\beta\times \R$, and $R=\gamma\times \R$. We see immediately that \[d^\pi_C(A,B)=d_\gamma(\alpha,\beta)\] where $d_\gamma(\alpha,\beta)$ denotes the distance between the closest point on $\gamma$ to $\alpha$ and the closest point on $\gamma$ to $\beta$. There are finitely many possibilities for the vertex $u'\in [v,w]$. Thus to show that there are finitely many domains $C$ of type (ii) with $d^\pi_C(A,B)>\theta$, it suffices to show that there are finitely many boundary components $\gamma$ of $H$ with $d_\gamma(\alpha,\beta)>\theta$ (note that $\alpha$ and $\beta$ are uniquely determined by $A$ and $B$).

Let $[p,p']$ be the shortest geodesic from $\alpha$ to $\gamma$ (with $p\in \alpha$ and $p'\in \gamma$) and let $[q,q']$ be the shortest geodesic from $\beta$ to $\gamma$ (with $q\in \beta$ and $q'\in \gamma$). Also let $[r,s]$ be the shortest geodesic from $\alpha$ to $\beta$ (with $r\in \alpha$ and $s\in \beta$). Orient $[p,p']$ from $p$ to $p'$, $[q,q']$ from $q$ to $q'$, and $[p',q']\subset \gamma$ from $p'$ to $q'$ (see Figure \ref{fig:P1P2}). We claim that $d_\gamma(\alpha,\beta)>\theta$ implies that $\gamma$ contains a point which is $\epsilon$-close to $[r,s]$.  

\begin{figure}
\def\svgwidth{2.5in}  
  \centering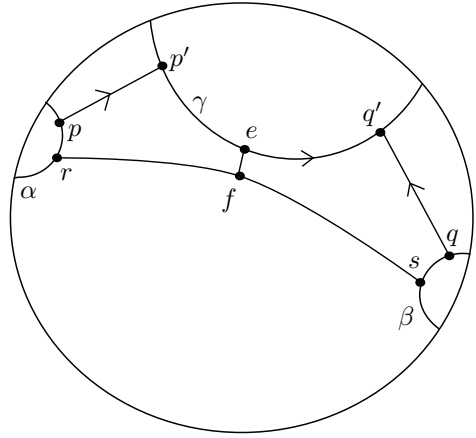 \\
	\caption{Proofs of Axioms (P1) and (P2).} 
	\label{fig:P1P2}
\end{figure}

By the definition of $[p,p']$ as the shortest geodesic from $\alpha$ to $\gamma$, the only points on $[p,p']$ which are $\epsilon$-close to $[p',q']\subset \gamma$ are the points in the final segment of length $\epsilon$. A similar statement holds for $[q,q']$. By the triangle inequality, the only points on $[p',q']$ which may be $\epsilon$--close to $[p,p']$ are the points in the \textit{initial} subsegment of length $2\epsilon$. Similarly, the only points on $[p',q']$ which may be $\epsilon$--close to $[q,q']$ are the points in the \textit{final} subsegment of length $2\epsilon$. Therefore any point in the middle segment of $[p',q']$ of length $(d(p',q')-4\epsilon)$ is $\epsilon$--close to a point of $[r,p]\cup [r,s]\cup [s,q]$. Note that $[r,p]\subset \alpha$ and $[s,q]\subset \beta$. Hence, only a segment of $[p',q']$ of length at most $\eta$ may be $\epsilon$--close to $[r,p]$ and only a segment of $[p',q']$ of length at most $\eta$ may be $\epsilon$--close to $[s,q]$. Therefore if $d(p',q')=d_\gamma(\alpha,\beta)>\theta=4\epsilon+2\eta$ then $[p',q']$ must have a point at distance at most $\epsilon$ from $[r,s]$, as claimed.

Finally, note that there are finitely many boundary components of $H$ which meet the $\epsilon$--neighborhood of $[r,s]$. This proves that there are finitely many boundary components $\gamma$ of $H$ with $d_\gamma(\alpha,\beta)>\theta$. As noted before, there are finitely many choices for the vertex $u'$, so this proves that there are finitely many domains $C$ with $d^\pi_C(A,B)>\theta$.

Finally, we prove Axiom (P1). If $d^\pi_C(A,B)>0$ then the vertex $u$ corresponding to $C$ must have one of the types (i) or (ii) above. In other words, if $v, w,$ and $u$ are the vertices for $A, B,$ and $C$ respectively, and $[v,w]$ is the geodesic between $v$ and $w$ in $\tilde{\Gamma}$, then either (i) u lies on $[v,w]$ or (ii) $u$ is joined by an edge to a vertex $u'\in [v,w]$. In case (i) we have \[d^\pi_A(B,C)=d^\pi_B(A,C)=0\] so that (P1) is trivially satisfied.

In case (ii) we again have that \[d^\pi_A(B,C)=d^\pi_B(A,C)=0\] unless $u'$ is joined by an edge to $v$ or to $w$. Suppose for instance that $u'$ is joined by an edge to $v$. Then parametrizing $\tilde{M}_{u'}$ as $H\times \R$ with $H$ a closed convex subset of $\hyp^2$, we see that there are boundary components $\alpha,\beta,\gamma$ of $H$ with \[d^\pi_C(A,B)=d_\gamma(\alpha,\beta) \text{ and } d^\pi_A(B,C)=d_\alpha(\beta,\gamma).\] We claim that if $d_\gamma(\alpha,\beta)>\theta$ then $d_\alpha(\beta,\gamma)\leq \theta$.

By the proof of Axiom (P2), if $d_\gamma(\alpha,\beta)>\theta$, then we have the following property. Let $[r,s]$ be the shortest geodesic from $\alpha$ to $\beta$, $[p,p']$ be the shortest geodesic from $\alpha$ to $\gamma$, and $[q,q']$ the shortest geodesic from $\beta$ to $\gamma$. Then we have that $[p',q']\subset \gamma$ has a point which is $\epsilon$--close to $[r,s]$.

Now assume that $d_\alpha(\beta,\gamma)>\theta$. Consider points $e\in [p',q']$ and $f\in [r,s]$ such that $d(e,f)\leq \epsilon$ (see Figure \ref{fig:P1P2}). Then the geodesic pentagon \[[r,p]\cup [p,p']\cup [p',e]\cup [e,f]\cup [f,r]\] is $\epsilon$--thin. Orient each of $[r,p]\subset \alpha$, $[r,f]\subset [r,s]$, and $[p,p']$ to point from the first point in the brackets to the second. By the definition of $[r,s]$ as the shortest geodesic from $\alpha$ to $\beta$, only points on the initial subsegment of $[r,f]$ of length $\epsilon$ can be $\epsilon$--close to $\alpha$. Similarly, only points on the initial subsegment of $[p,p']$ of length $\epsilon$ can be $\epsilon$--close to $\alpha$. By the triangle inequality, only points in the initial subsegment of $[r,p]$ of length $2\epsilon$ can be $\epsilon$--close to $[r,f]$ and only points on the final subsegment of $[r,p]$ of length $2\epsilon$ can be $\epsilon$--close to $[p,p']$. Thus, points on the middle subsegment of $[r,p]$ of length $(d(r,p)-4\epsilon)$ are $\epsilon$--close to $[p',e]\cup [e,f]$. Since $[e,f]$ has length $\leq \epsilon$, any point on the middle subsegment of $[r,p]$ of length $(d(r,p)-4\epsilon)$ is $2\epsilon$--close to $[p',e]\subset \gamma$. Since $d_\alpha(\beta,\gamma)=d(r,p)>\theta=4\epsilon+2\eta$, this implies that $\alpha$ is $2\epsilon$--close to $\gamma$ along a segment of length at least $2\eta$. This contradicts the definition of $\eta$. Thus we must have $d_\alpha(\beta,\gamma)\leq \theta$.

Therefore \[d^\pi_C(A,B)> \theta \text{ implies } d^\pi_A(B,C)\leq \theta \text{ and } d^\pi_B(A,C)\leq \theta,\] as desired.
\end{proof}

\subsection{The case that $\Gamma$ contains loops}
\label{sec:loops}

Now we assume that $\Gamma$ contains loops but $M$ contains at least two \textit{distinct} pieces $X$ and $Y$ which are glued together (recall that we do not cover the case that $\Gamma$ has only one vertex). Thus $\Gamma$ contains an edge which is not a loop, between the vertex corresponding to $X$ and the vertex corresponding to $Y$. We again find elements $a$ and $b$ of $\pi_1(M)$, corresponding to orthogonal loops in the torus along which $X$ and $Y$ are glued such that $a$ and $b$ commute and $a$ and $b$ both fix a pair of lifts $\tilde{X}_0$ and $\tilde{Y}_0$ which are glued along a Euclidean plane in $\tilde{M}$. We assume that $a$ and $b$ act on $\tilde{X}_0$ and $\tilde{Y}_0$ in the same way as in the last section. We again claim that there is a hyperbolic space $\mathcal{L}$ on which $\pi_1(M)$ acts so that $a$ is elliptic while $b$ is loxodromic. However, the construction in this case is more complicated, and the space $\mathcal{L}$ will actually be a quasi-line.

If there is no loop in $\Gamma$ based at the vertex corresponding to $X$ then the techniques in the previous section apply, so that we obtain a hyperbolic space acted on by $\pi_1(M)$ with $a$ elliptic and $b$ loxodromic. Hence we suppose here without loss of generality that there is at least one loop in $\Gamma$ based at the vertex corresponding to $X$. In this case, taking $\mathbf{X}$ to be the set of \textit{all} lifts of the vertex space $X$ in $M$ leads to pairs of domains between which no projection is defined. Thus we instead use a coloring of the vertices of $\tilde{\Gamma}$ by two colors, black and white. As in the usual definition of a coloring, we require that if two vertices are joined by an edge, then they have different colors.

If $A$ and $B$ are distinct lifts of $X$ to $\tilde{M}$ and both correspond to \textit{black} vertices of $\tilde{\Gamma}$ (that is $A=\tilde{M}_v$ and $B=\tilde{M}_w$ and $v$ and $w$ are both black) then the geodesic $[v,w]\subset \tilde{\Gamma}$ contains at least one vertex in its interior. Thus there is a well-defined projection from $\ell_A$ to $\ell_B$ as before, with image a point. Hence we take our set of domains to be \[\mathbf{X}=\{A: A \text{ is a lift of } X \text{ and } A=\tilde{M}_v \text{ such that } v\in \tilde{\Gamma}^0 \text{ is black}\}.\] We refer to the domains $A\in{\bf X}$ as simply \textit{black} lifts of $X$. As in the previous subsection, for any sufficiently large $K$ we may define a quasi-tree $\mathcal{C}_K(\mathbf{X})$.

Now, $\pi_1(M)$ acts by permutation on the set of colors of vertices and thus an index two (normal) subgroup $N$ of $\pi_1(M)$ preserves the coloring of $\tilde{\Gamma}$. We have an action of $N$ on $\mathcal{C}_K(\mathbf{X})$, for a constant  $K$ which will be chosen in the course of Lemma \ref{qmvalues}. Moreover, both $a$ and $b$ lie in $N$. We may choose the coloring of $\Gamma$ such that the chosen lift $\tilde{X}_0$ corresponds to a black vertex. Then $a$ and $b$ both fix the domain $\tilde{X}_0$. We have that $a$ acts on $\mathcal{C}(\tilde{X}_0)=\ell_{\tilde{X}_0}$ by fixing every point and $b$ acts on $\mathcal{C}(\tilde{X}_0)$ as a translation. Hence $a$ is elliptic and $b$ is loxodromic in the action of $N$ on $\mathcal{C}(\mathbf{X})$. Note that the action of $N$ cannot be extended to an action of $\pi_1(M)$ since elements of $\pi_1(M)$ send domains in $\mathbf{X}$ to domains which have no projections defined to the domains of $\mathbf{X}$. We wish to construct a homogeneous quasimorphism $q_0:N\to \R$ such that $q_0(b)\neq 0$ and $q_0(a)=0$ by applying  Proposition \ref{prop:qmconstruction}.
%Recall Proposition \ref{prop:qmconstruction}:
%
%\qmconstruction*

\begin{lem}
The element $b$ is WWPD${}^+$ in the action of $\pi_1(M)$ on $\C_K(\mathbf{X})$.
\end{lem}

\begin{proof}
The fact that $b$ is WWPD follows  from the quasi-tree of metric spaces machinery. Any conjugate of $b$ has a geodesic axis in $\C_K(\mathbf{X})$. If this axis is not the same as the axis of $b$, which is $\ell_{\tilde{X}_0}$, then it is equal to $\ell_A$ for some $A\in \mathbf{X}\setminus \{\tilde{X}_0\}$. Then $\pi_{\tilde{X}_0}(A)$ is a point and the closest point projection of $\ell_A$ to $\ell_{\tilde{X}_0}$ in $\C_K(\mathbf{X})$ is a uniformly bounded diameter set at a uniformly bounded distance from $\pi_{\tilde{X}_0}(A)$ (see Theorem \ref{thm:isomembed}). Thus the projection of a translate of the axis of $b$ to the axis of $b$ has uniformly bounded diameter.

To show that $b$ is WWPD${}^+$, we must show that no element of $N$ interchanges the endpoints of $\ell_{\tilde{X}_0}$. If an element of $N$ fixes $\ell_{\tilde{X}_0}$ then it also fixes the domain $\tilde{X}_0$. The stabilizer of $\tilde{X}_0$ is conjugate to $\pi_1(S)\times \Z$ where $X=S\times S^1$. Clearly, no element of this stabilizer interchanges the endpoints of $\ell_{\tilde{X}_0}$.
\end{proof}

Hence by Proposition \ref{prop:qmconstruction} there is a homogeneous quasimorphism $q_0\colon N\to \R$ with $q_0(a)=0$ and $q_0(b)\neq 0$. We will use this to define a quasimorphism $\pi_1(M)\to\R$ by using a technique from the proof of \cite[Lemma~7.2]{symmetric}.  In order to apply this technique, we must first modify $q_0$ to a quasimorphism $q_0'\colon N\to \R$ as follows. Choose $h$ to be a representative of the nontrivial coset of $N$ in $\pi_1(M)$ and define \[q_0'(g)=q_0(g)+q_0(h g h^{-1}).\] Since $N$ is normal, $q_0'$ is indeed a map $N\to \R$. Moreover, $q_0'$ extends to a homogeneous quasimorphism $q\colon \pi_1(M)\to \R$ defined by \[q(g)=\frac{1}{2}q_0'(g^2) \text{ for any } g\in \pi_1(M).\] (In other words, $q|_N=q_0'$.) The fact that $q$ is a quasimorphism is given in the proof of \cite[Lemma~7.2]{symmetric}, and it is straightforward to check that it is homogeneous. Our main tool in this section is the following:

\begin{lem}
\label{qmvalues}
For $K$ large enough, $q(a)=0$ and $q(b)\neq 0$.
\end{lem}

Before giving the proof of this lemma, we will show how it proves our main result. 
\begin{thm}
\label{loops}
The group $\pi_1(M)$ admits no largest action.
\end{thm}

\begin{proof}[Proof of Theorem \ref{loops} using Lemma \ref{qmvalues}]
By Lemma \ref{lem:quasiline}, the quasimorphism $q$ gives rise to an action of $\pi_1(M)$ on a quasi-line $\mathcal{L}$. In this action, $a$ is elliptic and $b$ is loxodromic. Reversing the roles of $a$ and $b$, we find an action of $\pi_1(M)$ on a quasi-line $\mathcal{L}'$ such that $a$ is loxodromic and $b$ is elliptic.
Applying Lemma \ref{lem:mainlem} completes the proof.
\end{proof}

The remainder of this section is devoted to the proof of Lemma \ref{qmvalues}.  Since $h$ is a representative of the nontrivial coset of $N$, it interchanges the colors of $\tilde{\Gamma}$. Hence $h a h^{-1}$ fixes a lift $W$ of $X$ which corresponds to a vertex $v$ of $\tilde{\Gamma}$ which is \textit{white}. Parametrizing $W=H_1\times \R$, with $H_1$ a closed convex subset of $\hyp^2$, we have that $a$ acts on $W$ as $\phi\times \operatorname{id}$ where $\phi$ is a loxodromic isometry of $H_1$ with axis equal to a boundary component $\beta$ of $H_1$. The lift $W$ is glued to a lift $\tilde{Y}$ of $Y$. The lift $\tilde{Y}$ may be written as $H_2\times \R$ and $W$ and $\tilde{Y}$ are glued along the Euclidean plane $\beta\times \R \subset \partial W$ which is identified with a component $\gamma\times \R$ of $\partial \tilde{Y}$, where $\gamma\subset \partial H_2$.

The element $h a h^{-1}$ permutes the infinitely many lifts of $X$ adjacent to $W$. Call them $B_1,B_2,\ldots$. Note that they are all \textit{black}.

\begin{lem} \label{lem:distbound}
The distances $d^\pi_C(h a^m h^{-1} B_1,h a^n h^{-1} B_1)$ are bounded for any $m,n\in \Z$ and $C$ a black lift of $X$ not contained in $\{h a^m h^{-1} B_1,h a^n h^{-1} B_1\}$.
\end{lem}

\begin{proof}
We first handle the case that $C \in \{B_1,B_2,\ldots\}$. For each $i$, there is a boundary component $\alpha_i$ of $H_1$ such that $B_i$ is glued to $W$ along the boundary plane $\alpha_i\times \R$.

We see that $d^\pi_{B_i}(B_j,B_k)=d_{\alpha_i}(\alpha_j,\alpha_k)$ for all $i,j,k$. In particular, $d^\pi_{B_i}(h a^m h^{-1} B_1,h a^n h^{-1} B_1)=d_{\alpha_i}(h a^m h^{-1} \alpha_1,h a^n h^{-1} \alpha_1)$.

\begin{figure}[h]
\centering

\def\svgwidth{0.5\textwidth}
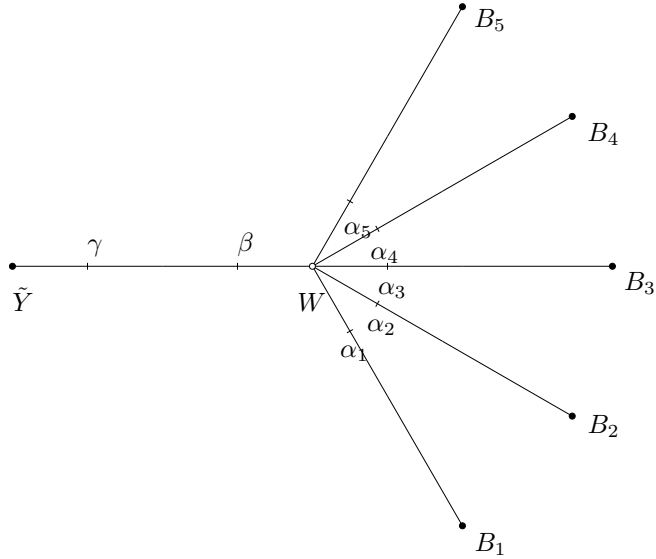
\caption{The $\alpha_i$ represent boundary components in $H_1$ where $W=H_1\times \R$. The $\alpha_i$ also represent the vertical directions in the $B_i$. Hence we see that $d^\pi_{B_i}(B_j,B_k)=d_{\alpha_i}(\alpha_j,\alpha_k)$.}
\label{halfedges}
\end{figure}

If $\alpha_i$ is not equal to $h a^k h^{-1} \alpha_1$ for any $k$ then we see that there exists a unique $k$ such that $\alpha_i$ lies between $h a^k h^{-1} \alpha_1$ and $h a^{k+1}h^{-1} \alpha_1$ (see Figure \ref{xiprojdist}). Denote by $\pi_{\alpha_i}(\alpha_j)$ the nearest point to $\alpha_j$ on $\alpha_i$. Fixing an appropriate orientaton on $\alpha_i$, we see that the projections $\pi_{\alpha_i}(h a^l h^{-1} \alpha_1)$ of the $h a^l h^{-1} \alpha_1$ onto $\alpha_i$ occur in the order \[\pi_{\alpha_i}(h a^k h^{-1}\alpha_1) < \pi_{\alpha_i}(h a^{k-1} h^{-1} \alpha_1) < \pi_{\alpha_i}(h a^{k-2} h^{-1} \alpha_1) < \cdots <\pi_{\alpha_i}(\beta)\] and \[\pi_{\alpha_i}(h a^{k+1} h^{-1} \alpha_1) > \pi_{\alpha_i}(h a^{k+2} h^{-1} \alpha_1) >\pi_{\alpha_i}(h a^{k+3} h^{-1} \alpha_1) > \cdots > \pi_{\alpha_i}(\beta).\] Hence $d_{\alpha_i}(h a^m h^{-1} \alpha_1,h a^n h^{-1} \alpha_1)$ is bounded by $d_{\alpha_i}(h a^k h^{-1} \alpha_1,h a^{k+1} h^{-1} \alpha_1)$. This projection distance is in turn bounded from above only in terms of the distance $d(h a^k h^{-1} \alpha_1,h a^{k+1} h^{-1} \alpha_1)$ from $h a^k h^{-1} \alpha_1$ to $h a^{k+1} h^{-1} \alpha_1$ in $\hyp^2$. Of course, we have $d(h a^k h^{-1} \alpha_1,h a^{k+1} h^{-1} \alpha_1)=d(\alpha_1,h a h^{-1} \alpha_1)$ is independent of $k$. Therefore $d_{\alpha_i}(h a^m h^{-1} \alpha_1,h a^n h^{-1} \alpha_1)$ is bounded above independently of $i$, $m$, and $n$ as long as $\alpha_i$ is not equal to any $h a^kh^{-1} \alpha_1$. 

\begin{figure}[h]

\centering
\def\svgwidth{0.5\textwidth}
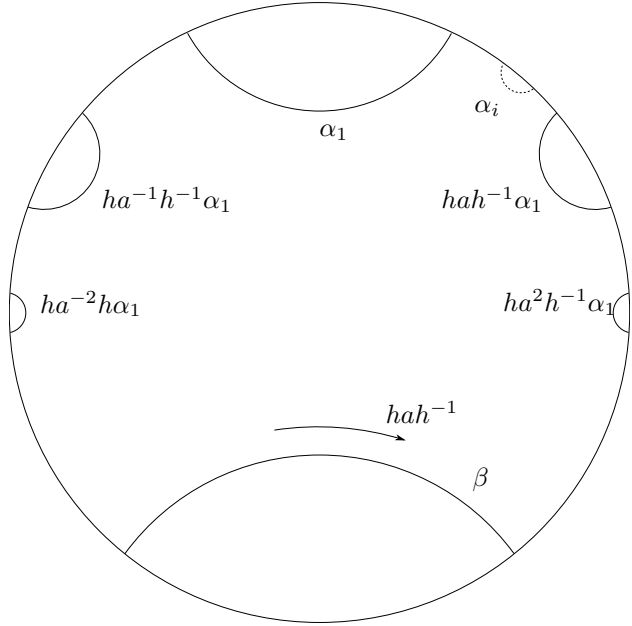

\caption{If $\alpha_i$ is between $ha^kh^{-1}\alpha_1$ and $ha^{k+1}h^{-1}\alpha_1$ then $d_{\alpha_i}(h a^m h^{-1} \alpha_1,h a^n h^{-1} \alpha_1)$ is bounded by $d_{\alpha_i}(h a^k h^{-1} \alpha_1,h a^{k+1} h^{-1} \alpha_1)$ (in this picture $k=0$).}
\label{xiprojdist}
\end{figure}

If $\alpha_i$ is equal to $h a^k h^{-1} \alpha_1$ for some $k$ with $k\notin \{m,n\}$, then we see as above that $d_{h a^k h^{-1} \alpha_1}(h a^m h^{-1} \alpha_1,h a^n h^{-1} \alpha_1)$ is bounded by $d_{h a^k h^{-1} \alpha_1}(h a^{k-1} h^{-1} \alpha_1,h a^{k+1} h^{-1} \alpha_1)$. This is in turn bounded above in terms of $d(h a^{k-1}\alpha_1 h^{-1} ,h a^{k+1} h^{-1} \alpha_1)=d(\alpha_1,h a^2 h^{-1} \alpha_1)$, which is clearly independent of $k,m,$ and $n$.

Finally if $C\notin \{B_1,B_2,\ldots\}$ then we claim that $d^\pi_{C}(h a^m h^{-1} B_1,h a^n h^{-1} B_1)=0$. For in this case let $u$ be the (black) vertex of $\tilde{\Gamma}$ corresponding to $h a^m h^{-1} B_1$, let $v$ be the vertex corresponding to $h a^n h^{-1} B_1$, and let $w$ be the vertex corresponding to $C$. Consider the geodesic $[v,w]$ from $v$ to $w$ in $\tilde{\Gamma}$. Up to exchanging the roles of $v$ and $w$, the geodesic $[u,u']$ from $u$ to $[v,w]$ has one of the two following properties: either (i) $u'=v$ and $[u,u']$ has length two or (ii) $u'$ is the first vertex on $[v,w]$ \textit{after} $v$ and $[u,u']$ has length one.   In the first case we have automatically that $\pi_{C}(h a^m h^{-1} B_1)=\pi_{C}(h a^n h^{-1} B_1)$. In the second case, we see that $u'$ is the vertex corresponding to the chosen lift $W$. We again see that $\pi_{C}(h a^m h^{-1} B_1)= \pi_{C}(h a^n h^{-1} B_1)$ unless $[v,w]$ has length two. In this case the lift $C$ is adjacent to $W$ and therefore $C=B_i$ for some $i$, which is a contradiction.
\end{proof}

\begin{proof}[Proof of Lemma \ref{qmvalues}]
First we prove that $q(a)=0$.

We have \[q(a)=q_0'(a)=q_0(a)+q_0(h ah^{-1}).\] Since, $a$ is elliptic, Proposition \ref{prop:qmconstruction} implies that  $q_0(a)=0$. We check that also $q_0(h a h^{-1})=0$, and this will prove that $q(a)=0$, as desired.

As in \cite[Section~3.2]{bbf}, we modify the distance functions $d^\pi_{C}$ to distance functions $d_{C}$, which satisfy $d_{C}\leq d^\pi_{C}$.   Choose $K$ large enough that $d^\pi_{C}(h a^m h^{-1} B_1,h a^n h^{-1} B_1)$ is bounded by $K$ for all $m,n\in \Z$ and for any $C\notin \{h a^m h^{-1} B_1,h a^n h^{-1} B_1\}$; such a $K$ exists by Lemma \ref{lem:distbound}. We build the quasi-tree of metric spaces $\mathcal{C}_K(\mathbf{X})$ using the spaces $\mathcal{C}(C)=\ell_C$, projections $\pi_C$, and distances $d_C$. Since for all $m,n$ and $C\notin \{h a^mh^{-1} B_1,h a^n h^{-1} B_1\}$, we have \[d_{C}(h a^m h^{-1} B_1,h a^n h^{-1} B_1)\leq d^\pi_{C}(h a^m h^{-1} B_1,h a^n h^{-1} B_1)\leq K,\] the space $\mathcal{C}(h a^m h^{-1} B_1)$ is joined by an edge to the space $\mathcal{C}(h a^n h^{-1} B_1)$ for all $m,n\in \Z$. Furthermore, this edge goes from $\pi_{h a^m h^{-1} B_1}(h a^n h^{-1} B_1)$ to $\pi_{h a^n h^{-1} B_1}(h a^m h^{-1} B_1)$.

To see that $h a h^{-1}$ is elliptic, consider as a basepoint $P=\pi_{B_1}(h a h^{-1} B_1)$. For $i\in \Z\setminus \{0\}$, let $e_i$ be the edge joining $\pi_{B_1}(h a^i h^{-1} B_1)$ to $\pi_{h a^i h^{-1} B_1}(B_1)$. The endpoints of the $e_i$ on $\mathcal{C}(B_1)$ occur between those of $e_{-1}$ and $e_1$. Furthermore the distance between the endpoints of $e_1$ and $e_{-1}$ on $\mathcal{C}(B_1)$ is bounded by $K$. Note that the endpoint of $e_1$ on $\mathcal{C}(B_1)$ is $P$. We now give an upper bound on the distance from $P$ to $h a^k h^{-1} P$ in $\mathcal{C}_K(\mathbf{X})$, which is independent of $k$. We have $h a^k h^{-1} P=\pi_{h a^k h^{-1} B_1}(h a^{k+1} h^{-1} B_1)$ and this is the endpoint on $\mathcal{C}(h a^k h^{-1} B_1)$ of the edge $h a^k h^{-1}e_1$. The endpoints of the edges $h a^k h^{-1} e_i$ on $\mathcal{C}(h a^k h^{-1} B_1)$ occur between the endpoints of $h a^k h^{-1} e_{-1}$ and $h a^k h^{-1} e_1$ and these are distance at most $K$ apart. Finally, the edge $e_k$ from $\mathcal{C}(B_1)$ to $\mathcal{C}(h a^k h^{-1}B_1)$ is equal to $h a^k h^{-1} e_{-k}$, which, by definition of $\C_K({\bf X})$, has length $L$. Thus, we have \[d(P,h a^k h^{-1} P)\leq d(P,e_k)+\operatorname{length}(e_k)+d(e_k,h a^k h^{-1}P)\leq K+L+K.\] See Figure \ref{elliptic}. This completes the proof that $h a h^{-1}$ is elliptic.

\begin{figure}[h]
\centering
\def\svgwidth{0.8\textwidth}
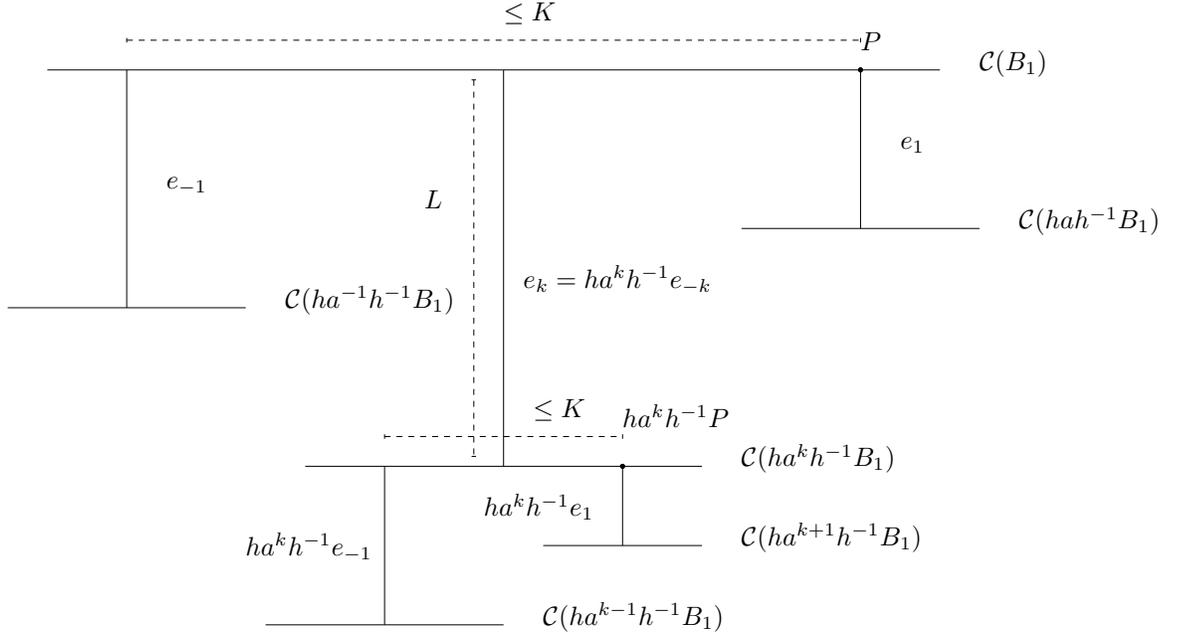

\caption{The arrangements of various lines in the quasi-tree $\mathcal{C}_K(\mathbf{X})$, showing that the orbit of the point $P$ is bounded.}
\label{elliptic}
\end{figure}

Finally, we show that $q(b)=q_0'(b)\neq 0$. We have \[q_0'(b)=q_0(b)+q_0(h b h^{-1})\] and $q_0(b)\neq 0$. We will show that $h b h^{-1}$ is elliptic in the action $N\curvearrowright \mathcal{C}_K(\mathbf{X})$. This will show that $q_0(h bh^{-1})=0$ and thus $q_0'(b)=q_0(b)\neq 0$, as claimed. Note that the conjugate $h bh^{-1}$ fixes the lift $W=h\tilde{X}_0$ of $X$ which is white, on which it acts as a vertical translation.  That is, parametrizing $W=H_1\times \R$ with $H_1$ a closed convex subset of $\mathbb H^2$, we have that $h bh^{-1}$ acts as $\textrm{id}\times \phi$ where $\phi$ is a loxodromic isometry of $\R$. As before, let $B_1,B_2,\ldots$ be the other lifts of $X$ to which $W$ is glued. These all correspond to \textit{black} vertices of $\tilde{\Gamma}$ and are fixed by $h bh^{-1}$.  Parametrizing $B_i=H'_i\times\R$, with $H'_i$ a closed convex subset of $\mathbb H^2$, we see that the vertical directions of $B_i$ correspond to boundary components of $H_1$.  Thus $h bh^{-1}$ fixes each $\ell_{B_i}$ pointwise, hence $h bh^{-1}$ is elliptic, as claimed.
\end{proof}

\section{Fundamental groups of finite-volume cusped hyperbolic 3-manifolds}
\label{sec:hyp3man}

Let $M$ be a finite-volume cusped hyperbolic 3-manifold. The \textit{cusps} of $M$ are all homeomorphic to $T^2\times [0,\infty)$ where $T^2$ is the torus. We recall that \textit{Dehn filling} one of the cusps of $M$ consists of 
\begin{itemize}
\item removing $T^2\times (0,\infty)$ to form the manifold $M_0$ with a boundary component homeomorphic to $T^2$,
\item gluing in a copy of the solid torus $D^2\times S^1$ by identifying its boundary with the boundary component of $M_0$ by some homeomorphism to define the manifold $M'$.
\end{itemize}
This homeomorphism identifies $\partial D^2 \times \{*\}$ (where $*$ is any point of $S^1$) with a simple closed curve on the torus boundary component of $M_0$. The isotopy class of this curve is called the \textit{slope} of the Dehn filling, and it can be shown that the resulting manifold $M'$ is determined up to homeomorphism by this slope.

Since $M$ has finite volume, it has finitely many cusps $T_1\times [0,\infty),\ldots,T_n\times [0,\infty)$. Thurston's Hyperbolic Dehn Surgery Theorem states:

\begin{thm}[{\cite{thurston}}]
There exists a finite set $S_i$ of slopes on $T_i$ such that if we Dehn fill \textit{each} cusp of $M$ and the slope for $T_i$ avoids $S_i$ for each $i$, then the resulting manifold $M'$ admits a hyperbolic metric.
\end{thm}

The Dehn-filled manifold $M'$ in this theorem is closed. If we choose the slope $\alpha_i$ on $T_i$ for each $i$, then $\alpha_i$ corresponds to a conjugacy class in $\pi_1(M)$. We have \[\pi_1(M')=\pi_1(M)/ \llangle\alpha_1, \ldots, \alpha_n  \rrangle\] where $\llangle  S \rrangle $ denotes the normal closure of a subset $S$ in $\pi_1(M)$. In particular, there is a quotient $\pi_1(M)\onto \pi_1(M')$. Any isotopy class of simple closed curve on $T_i$ other than $\alpha_i$ corresponds to a conjugacy class in $\pi_1(M)$ which is \textit{not} in the kernel of the quotient map $\pi_1(M)\to \pi_1(M')$.

Consider the cusp $T_1\times [0,\infty)$. By choosing as a basepoint $x\in T_1\times \{0\}$, loops in $T_1\times \{0\}$ based at $x$ define a subgroup $H\cong \Z^2$ of $\pi_1(M,x)$. A slope on $T_1$ then corresponds to a primitive element of $\Z^2$. Choose two primitive elements $a$ and $b$ of $H$ which do not lie in the finite set $S_1\subset H$.  We are now ready to prove the main result of this section.

\begin{thm}
The poset $\H(\pi_1(M))$ contains no largest element.
\end{thm}

\begin{proof}
We will show that 
%\begin{lem}
%\label{lem:cuspactions}
there are cobounded actions $\pi_1(M)\curvearrowright X$ and $\pi_1(M)\curvearrowright Y$ with $X$ and $Y$ hyperbolic such that $a$ acts loxodromically and $b$ acts elliptically on $X$ and $a$ acts elliptically and $b$ acts loxodromically on $Y$.  The result will then follow by Lemma \ref{lem:mainlem}.
%\end{lem}

%We have as an immediate corollary:

%\begin{proof}[Proof of Lemma \ref{lem:cuspactions}]
Let $M'$ be the manifold obtained by Dehn filling $T_1\times [0,\infty)$ with slope $b$ and filling $T_2\times [0,\infty),\ldots, T_n\times [0,\infty)$ with \textit{any} slopes avoiding the sets $S_2,\ldots,S_n$. The resulting manifold $M'$ is closed and hyperbolic so $\pi_1(M')$ admits a cobounded properly discontinuous action on $\hyp^3$. We obtain an action of $\pi_1(M)$ on $\hyp^3$ by first taking the quotient $\pi_1(M)\to \pi_1(M')$ and then composing with the action of $\pi_1(M')$ on $\hyp^3$. Of course $b$ acts elliptically in this action since it lies in the kernel of $\pi_1(M)\to \pi_1(M')$. On the other hand, $a$ does not lie in this kernel. Since every nontrivial element of $\pi_1(M')$ acts loxodromically on $\hyp^3$, $a$ acts loxodromically in this action.  This gives us our action $\pi_1(M)\curvearrowright X$. Reversing the roles of $a$ and $b$ gives the construction of $\pi_1(M)\curvearrowright Y$.
\end{proof}

\begin{rem}
The proof above applies without major changes whenever $G$ is relatively hyperbolic with a peripheral subgroup isomorphic to $\Z^2$ or, even more generally, when $G$ is acylindrically hyperbolic with a \textit{hyperbolically embedded} subgroup isomorphic to $\Z^2$ using a more general version of Dehn filling (see \cite{dgo} for definitions and details). We omit the details.
\end{rem}

%\begin{thm}
%Any group $G$ such that $\Z^2\h G$ is $\mathcal H$--inaccessible.
%\end{thm} 
%
%
%\begin{proof}[Sketch of proof]
%Let $\Z^2=\langle a,b\rangle$, and let $G'=G/\llangle a^n\rrangle$ for some sufficiently large $n$.  Then by Dehn filling theory for acylindrically hyperbolic groups, we have that $\Z^2/\langle a^n\rangle \h G'$.  Moreover, $\Z^2/\langle a^n\rangle\simeq E(b)$, the virtually cyclic subgroup containing $\langle b\rangle$.  Thus, by Osin Theorem 1.4, there is an $[X]\in\AH(G')$ such $b\in \L([X])$.  We can extend this to an action $G\curvearrowright \Gamma(G,X)$, which may no  longer be acylindrical, so that $[X]\in\HG$, and $b\in \L([X])$ while $a^k\not\in\L([X])$ for all $k$.  Similarly, we can find $[Y]\in\HG$ such that $a\in \L([Y])$ while $b^k\not\in\L([Y])$ for all $k$.  Since $a$ and $b$ commute, it follows that $G$ is $\mathcal H$--inaccessible.
%
%\end{proof}
%
%\begin{cor}
%Relatively hyperbolic groups with a $\Z^2$ peripheral subgroup are not $\mathcal H$--accessible.  In particular, finite volume cusped hyperbolic $3$--manifolds are not $\H$--accessible.
%\end{cor}

\section{Fundamental groups of Anosov mapping tori}
\label{sec:anosov}

Fix an element $\phi\in \SL(2,\Z)$ that is Anosov; that is, $\phi$ has distinct irrational eigenvalues $\lambda>1$ and $\lambda^{-1}$. The map $\phi$ is an element of the mapping class group of the torus $T^2$, and the group $G=\Z^2 \rtimes_\phi \Z$ is the fundamental group of the mapping torus of $\phi$. Our goal in this section is prove Theorem \ref{anosovposet}.  To do so, we will follow the following steps:
\begin{enumerate}[(1)]
\item Since $G$ is solvable it admits only lineal and quasi-parabolic structures.
\item We show that quasi-parabolic structures are equivalent to confining subsets of $\Z^2$ under the action of $\phi$ or $\phi^{-1}$.
\item We classify quasi-parabolic structures using the correspondence with confining subsets and the geometry of $\R^2$, showing that there are only two up to equivalence.
\item Finally, we classify the lineal structures.
\end{enumerate}

We begin by considering the abelianization of $G$.

\begin{lem}
\label{abelianization}
The abelianization of $G$ is virtually cyclic.
\end{lem}

\begin{proof}
Writing $\phi = \begin{pmatrix} a & b \\ c & d \end{pmatrix},$ the group $G$ has the presentation \[G=\left\langle x, y, t : [x,y]=1, txt^{-1}=x^ay^c, tyt^{-1}=x^by^d\right\rangle.\] The abelianization $G'$ is generated by $\overline{x}$, $\overline{y}$, and $\overline{t}$, respectively, subject to the relations \[\overline{x}=a\overline{x}+c\overline{y}\quad \text{and} \quad \overline{y}=b\overline{x}+d\overline{y}.\] Equivalently, we may consider this as the system of equations \[ \begin{pmatrix} a-1 & c \\ b & d-1 \end{pmatrix} \begin{pmatrix} \overline{x} \\ \overline{y} \end{pmatrix} = \begin{pmatrix} 0 \\ 0 \end{pmatrix}.\] We will perform row reduction on this matrix, using only the following elementary row operations:
\begin{itemize}
\item adding an integer multiple of one row to another,
\item swapping rows, and
\item multiplying a row by an integer.
\end{itemize}
The result is a sequence of systems of equations which hold in $G'$. Since the matrix $\begin{pmatrix} a-1 & c \\ b & d-1 \end{pmatrix}$ has nonzero determinant, at the final stage we arrive at an equation \[\begin{pmatrix} m & 0 \\ 0 & n \end{pmatrix} \begin{pmatrix} \overline{x} \\ \overline{y}\end{pmatrix} = \begin{pmatrix} 0 \\ 0 \end{pmatrix},\] where $m,n\in \Z \setminus \{0\}$. In other words, $m\overline{x}=0$ and $n\overline{y}=0$. Thus, we have $G'\cong \langle \overline{t} \rangle \times \langle \overline{x}, \overline{y} \rangle$ and $\langle \overline{x}, \overline{y}\rangle$ is finite. This proves the statement.
\end{proof}

As in the proof of Lemma \ref{abelianization}, we denote by $t$ the generator of the $\Z$ factor of $G=\Z^2 \rtimes_\phi \Z$.

\begin{lem}
\label{lem:qpconfining}
A hyperbolic structure $[T]$ is an element of $\H_{qp}(G)$ if and only if there exists a symmetric subset $Q\subset \Z^2$ which is strictly confining under the action of $\phi$ or $\phi^{-1}$ such that $[T]=[Q\cup \{t^{\pm 1}\}]$.
\end{lem}

The proof relies on Lemma \ref{abelianization} and is completely analogous to \cite[Proposition~2.6]{ar},  which is in turn based on \cite[Theorems~4.4~\&~4.5]{ccmt}, so we omit it.

We consider symmetric subsets $Q\subset \Z^2$ which are confining under the action of $\phi$. Denote by $\lambda>1$ and $\lambda^{-1}$ the eigenvalues of $\phi$ with corresponding eigenvectors $v^+$ and $v^-$, respectively. We suppose that $v^+$ and $v^-$ have been chosen to be unit vectors. Given $\epsilon>0$, we define a symmetric subset of $\Z^2$ by \[Q_\epsilon=\{av^+ + bv^- \in \Z^2: a,b\in \R, \,|b|\leq\epsilon\}.\] In other words, $Q_\epsilon$ is the intersection of a neighborhood of the line $\R v^+$ in $\R^2$ with $\Z^2$.

\begin{figure}
\centering
\def\svgwidth{3in}
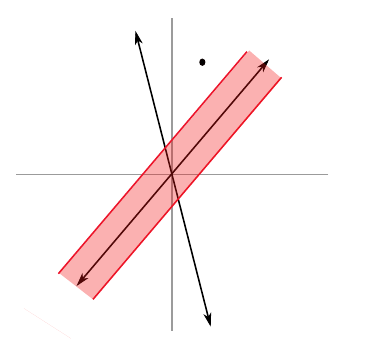
\caption{Proof of Lemma \ref{equivsets}.  The intersections of the blue and red shaded regions with $\mathbb Z^2$ are $Q_\epsilon$ and $Q_{\epsilon/\lambda}$, respectively.}
\label{fig:Q_varepsilon}
\end{figure}

\begin{lem}
\label{equivsets}
For any $\epsilon>0$, the set $Q_\epsilon$ is strictly $\phi$--confining. Furthermore if $\epsilon,\delta>0$, then we have \[[Q_\epsilon \cup \{t^{\pm 1}\}] = [Q_\delta \cup \{t^{\pm 1}\}].\]
\end{lem}

\begin{proof}
It is straightforward to check that $Q_\epsilon$ is confining using the fact that the action of $\phi$ on $\Z^2$ stretches $Q_\varepsilon$ in the direction of $v^+$ and contracts it in the direction of $v^-$. To prove that it is strictly confining, note that $\phi(Q_{\epsilon})\subset Q_{\epsilon/\lambda}$, and so it suffices to check that $Q_\epsilon \setminus Q_{\epsilon/\lambda}$ is nonempty. The set $\{av^+ + bv^- : a,b \in \R,\, |b|\leq\epsilon\}$ is bounded by two lines $L^u$ and $L^l$ in $\R^2$ which are parallel to $v^+$; see Figure \ref{fig:Q_varepsilon}. Similarly, $\{av^+ + bv^- : a,b\in \R,\, |b|\leq\epsilon/\lambda\}$ is bounded by two lines $M^u$ and $M^l$ which are parallel to $v^+$. If the labels are chosen such that $M^l$ is between $L^l$ and $M^u$ and $M^u$ is between $M^l$ and $L^u$, then we may choose a line $N$ parallel to $v^+$ between $M^u$ and $L^u$. The line $N$ passes arbitrarily close to the integer lattice $\Z^2$ since it projects to a line on $T^2$ which is dense in $T^2$. If we choose a point $p$ of $\Z^2$ which is sufficiently close to $N$, then $p\in Q_\epsilon \setminus Q_{\epsilon/\lambda}$, as desired.

For the second statement, suppose without loss of generality that $0<\delta\leq\epsilon$. Then $Q_\delta\subset Q_\epsilon$, and so $[Q_\delta\cup \{t^{\pm 1}\}]\succcurlyeq [Q_\epsilon \cup \{t^{\pm 1}\}]$. Choose $n\in \N$ large enough that $\epsilon/\lambda^n\leq\delta$. Then we have $\phi^n(Q_\epsilon)\subset Q2_\delta$. In other words, $t^nQ_\epsilon t^{-n}\subset Q_\delta$, and thus every element of $Q_\epsilon$ has word length at most $ 2n+1$ with respect to the generating set $Q_\delta\cup \{t^{\pm 1}\}$. Therefore $[Q_\delta \cup \{t^{\pm 1}\}] \preccurlyeq [Q_\epsilon \cup \{t^{\pm 1}\}]$, and the result follows.
\end{proof}

We next show that for some (equivalently, any) $\varepsilon>0$, the action of $G$ on $\Gamma(G,Q_\varepsilon\cup\{t^{\pm 1}\})$ is equivalent to the  action  of $G$ on $\hyp^2$ defined as follows. Let $\pi\colon \Z^2\to \R$ be the homomorphism \[\pi:av^+ + bv^-\mapsto b.\] Consider the upper half-plane model of $\hyp^2$. The group $\Z^2$ admits a parabolic action on $\hyp^2$ via \[p \cdot z=z+\pi(p), \text{ for } p\in \Z^2 \text{ and } z\in \hyp^2.\] Let $t$ act loxodromically on $\hyp^2$ by $t\cdot z=\lambda^{-1} z$. For $p\in \Z^2$ we have \[\phi(p)\cdot z=z+\pi(\phi(p))=z+\lambda^{-1}\pi(p).\] Moreover, \[tpt^{-1}\cdot z=tp\cdot \lambda z=t\cdot (\lambda z+\pi(p))=z+\lambda^{-1}\pi(p).\] Therefore the actions of $\Z^2$ and $\langle t \rangle$ induce an action of $G$ on $\hyp^2$.

\begin{lem}
\label{lem:hypplane}
The action of $G$ on $\hyp^2$ is equivalent to the action of $G$ on $\Gamma(G,Q_\epsilon \cup \{t^{\pm 1}\})$ for some (hence any) $\epsilon>0$.
\end{lem}

\begin{proof}
We apply the Schwarz-Milnor Lemma (\cite[Lemma~3.11]{abo}). Choose as a basepoint $i\in \hyp^2$. We claim that the $G$--translates of the ball $B=B_{\log(\lambda)}(i)$ cover $\hyp^2$. To see this, note that since $\R v^+$ passes arbitrarily close to the integer lattice $\Z^2$, the image $\pi(\Z^2)$ is dense in $\R$. Moreover, $t^n$ translates the horocycle $\{z\in \hyp^2 : \operatorname{Im}(z)=1\}$ to the horocycle $\{z\in \hyp^2: \operatorname{Im}(z)=\lambda^{-n}\}$ for any $n\in \Z$. Thus, for any $n\in \Z$, the orbit of $i$ is dense in the horocycle $\{z\in \hyp^2: \operatorname{Im}(z)=\lambda^{-n}\}$. These horocycles are spaced at distances exactly $\log(\lambda)$ apart. Hence, we easily see that the balls of radius $\log(\lambda)$ based at points in the orbit of $i$ cover $\hyp^2$.

By \cite[Lemma~3.11]{abo}, the action of $G$ on $\hyp^2$ is equivalent to the action of $G$ on $\Gamma(G,S)$ where \[S=\{g\in G : d(i,gi)\leq 2\log(\lambda)+1\}.\] Hence it remains to show that $[S]=[Q_\epsilon \cup \{t^{\pm 1}\}]$ for some $\epsilon$.

Suppose that $g\in S$, so that $g$ translates $i$ a distance of at most $2\log(\lambda)+1$. Writing $g=pt^n$ where $p\in \Z^2$ and $n\in \Z$, we have \[gi=\lambda^{-n}i+\pi(p).\] Hence, $gi$ lies on the horocycle $\{z\in \hyp^2:\operatorname{Im}(z)=\lambda^{-n}\}$. Since this horocycle has distance $|n|\log(\lambda)$ from $\{z\in \hyp^2:\operatorname{Im}(z)=1\}$, we must have \[|n|\log(\lambda)\leq 2\log(\lambda)+1,\] and therefore $|n|\leq \frac{2\log(\lambda)+1}{\log(\lambda)}$. We also have \[d(i,gi)=2\operatorname{arcsinh}\left( \frac{1}{2} \sqrt{\frac{\pi(p)^2+(\lambda^{-n}-1)^2}{\lambda^{-n}}}\right) \geq 2\operatorname{arcsinh} \left(\frac{1}{2} \sqrt{\frac{\pi(p)^2}{\lambda^{-n}}}\right).\] Since $\operatorname{arcsinh}$ is an increasing function and $|n|$ is bounded, this clearly defines an upper bound on $|\pi(p)|$. Set $\epsilon$ to be this upper bound. Since $|\pi(p)|\leq \epsilon$, we have $p \in Q_\epsilon$. Hence the word length of $g$ with respect to $Q_\epsilon \cup \{t^{\pm 1}\}$ is at most $1+\frac{2\log(\lambda)+1}{\log(\lambda)}$. This proves $[S] \succcurlyeq [Q_\epsilon \cup \{t^{\pm1}\}]$.

We now turn our attention to the other inequality.  Given $g\in Q_\epsilon \cup \{t^{\pm 1}\}$, we must consider two cases. If $g=t^{\pm 1}$ then $d(i,gi)=\log(\lambda)$, and therefore $g\in S$. On the other hand, if $g=p\in Q_\epsilon$ then we have \[d(i,pi)=2\operatorname{arcsinh} \left(\frac{1}{2} |\pi(p)|\right)\leq 2\operatorname{arcsinh} \left(\frac{1}{2} \epsilon\right).\] Letting $n$ be large enough so that $2\operatorname{arcsinh}\left(\frac{1}{2}\epsilon / \lambda^n \right)<2\log(\lambda)+1$, we have \[d(i,\phi^n(p)i)=2\operatorname{arcsinh} \left(\frac{1}{2}|\pi(p)/\lambda^n|\right)\leq 2\operatorname{arcsinh}\left(\frac{1}{2} \epsilon /\lambda^n\right)<2\log(\lambda)+1.\] Thus, $\phi^n(p)\in S$. As we already showed that $t^{\pm 1}\subset S$, it follows that $p=t^{-n}\phi^n(p)t^n$ has word length at most $2n+1$ with respect to $S$. This proves $[Q_\epsilon \cup \{t^{\pm 1}\}]\succcurlyeq [S]$.
\end{proof}

For the proof of the next lemma, denote by $\rho\colon\Z^2\to \R$ the homomorphism $\rho\colon av^++bv^-\mapsto a$.

\begin{lem}
\label{lem:containnbhd}
Let $Q\subset \Z^2$ be confining under the action of $\phi$. Then for some (hence any) $\epsilon>2$ we have \[[Q_\epsilon \cup \{t^{\pm 1}\}] \succcurlyeq [Q \cup \{t^{\pm 1}\}].\]
\end{lem}

\begin{proof}
Choose $n$ large enough that $\phi^n(Q+Q)\subset Q$ and $\lambda^n>2$; we fix this $n$ for the remainder of the proof.  Fix $u\in Q \setminus \{0\}$ and set $\epsilon=|b|$ where $u=av^++bv^-$. After possibly replacing $u$ by $-u$ we may assume $a>0$.

We define a sequence of subsets of $Q$ as follows. First of all define $P_0=\{\pm u \} \cup \{0\}$. Having defined $P_i$, we define $P_{i+1}$ inductively by $P_{i+1}=P_i \cup \phi^n(P_i+P_i)$. If we assume for induction that $P_i\subset Q$, we have $\phi^n(P_i+P_i)\subset \phi^n(Q+Q)\subset Q$ and therefore $P_{i+1}\subset Q$. Thus the union $P=\bigcup_{i=0}^\infty P_i\subset Q$. Moreover, $P$ is closed under the action of $\phi^n$ and we have $\phi^n(P+P)\subset P$.

Another induction argument shows that $P$ is contained in $Q_\epsilon$.  It is clear that $P_0\subset Q_\epsilon$. If we suppose for induction that $P_i\subset Q_\epsilon$ and $x,y\in P_i$, then we have $x=cv^+ + dv^-$ and $y=ev^++fv^-$, where $d$ and $f$ both have absolute value at most $\epsilon$. Thus $\phi^n(x+y)=\lambda^n(c+e)v^+ + \lambda^{-n}(d+f) v^-$. Since $\lambda^n>2$, we have  $|\lambda^{-n}(d+f)|< \frac{1}{2}(|d|+|f|)<\epsilon$. Therefore $P_{i+1}\subset Q_\epsilon$ as well.

Set $a=\rho(u)$ and for $r\in \N$ write $rP$ for the words in $P$ of length at most $r$.  Recall that a subset $S\subset \R$ is $R$--dense if for any $x\in \R$ there exists $s\in S$ with $|x-s|\leq R$.

\begin{claim} 
\label{denseprojs}
There exists $r>0$ such that $\rho(rP)$ is $(\lambda^n a)$--dense in $\R$.
\end{claim}

Before proving the claim, we show how it implies the lemma. Note that $rP$ is contained in $Q_{r\epsilon}$. We now claim that every element of $Q_{r\epsilon}$ may be written as a word in $Q$ of bounded word length. To see this, consider an element $g\in Q_{r\epsilon}$ and its projection $\rho(g)$ to $\R$. By Claim \ref{denseprojs}, there is an element $h\in rP$ with $|\rho(g)-\rho(h)|\leq \lambda^n a$. Furthermore, we have $d(g,\rho(g)v^+)\leq r\epsilon\|v^-\|$ and $d(h,\rho(h)v^+)\leq r\epsilon\|v^-\|$ (where $d(\cdot,\cdot)$ denotes distance in $\R^2$). Therefore  \[d(g,h)\leq d(g,\rho(g)v^+)+d(\rho(g)v^+,\rho(h)v^+)+d(\rho(h)v^+,h)\leq r\epsilon\|v^-\| + \lambda^n a+r\epsilon\|v^-\|=2r\epsilon\|v^-\|+\lambda^n a.\] Since $Q$ generates $\Z^2$, we may choose $N$ to be an upper bound on the word length in $Q$ of any element of $S=\Z^2\cap B_{2r\epsilon\|v^-\|+\lambda^n a}(0)$ (where $B_{2r\epsilon\|v^-\|+\lambda^n a}(0)$ is the ball of radius $2r\epsilon\|v^-\|+\lambda^n a$ centered at 0 in $\R^2$). Then since $d(g,h)\leq 2r\epsilon\|v^-\| +\lambda^na$, there exists $k\in S$ with $h+k=g$. We have \[||g||_Q\leq ||h||_Q+||k||_Q \leq r+N.\]

\begin{proof}[Proof of Claim \ref{denseprojs}]
Since $a=\rho(u)$, we have $\lambda^n a=\rho(\phi^n(u))$. Choose $r$ to be the largest integer with $r\lambda^n a<\lambda^{2n}a$; that is, $r=\lfloor \lambda^n \rfloor$. Then \[|\rho(r\phi^n(u)) - \rho(\phi^{2n}(u))| =  |r\lambda^na - \lambda^{2n}a| < \lambda^n a.\] 

We will show that the following hold for each $i\in \N$:
\begin{itemize}
\item $\rho(rP)$ intersects each interval $[\lambda^{in}a,\lambda^{(i+1)n}a]$ in a $(\lambda^n a)$--dense subset, and
\item $\rho(rP)$ contains the endpoints $\lambda^{in}a$ and $\lambda^{(i+1)n}a$.
\end{itemize}

Since $\rho(rP)$ is a symmetric subset of $\R$, this will prove that $\rho(rP)$ is $(\lambda^n a)$--dense in $\R$, as desired. The second claim is trivial since $\rho(\phi^{in}(u))=\lambda^{in}\rho(u)=\lambda^{in}a$ for any $i$. The first claim is proven by induction.

For the base case $i=1$, note that for $1\leq j\leq r$ we have $j\lambda^n a=\rho(j\phi^n(u))$. Since $j\leq r$ we have $j\phi^n(u)\in rP$. Since $|\lambda^{2n}a-r\lambda^n a|\leq \lambda^n a$, this proves that $\rho(rP)$ intersects $[\lambda^na,\lambda^{2n}a]$ in a $(\lambda^n a)$--dense set, as desired.

Now suppose for induction that we have a sequence $v_0,v_1, \ldots, v_t$ of elements of $rP$ with $\rho(v_0)=\lambda^{(i-1)n}a$, $\rho(v_t)=\lambda^{in}a$, and $|\rho(v_j)-\rho(v_{j+1})|\leq \lambda^n a$ for each $j$. We may apply $\phi^n$ to each element of the sequence. We have that $\rho(\phi^n(v_0))=\lambda^{in}a$, $\rho(\phi^n(v_t))=\lambda^{(i+1)n}a$, and, for each $i$, \[|\rho(\phi^n(v_j))-\rho(\phi^n(v_{j+1}))|=\lambda^n |\rho(v_j)-\rho(v_{j+1})|\leq \lambda^{2n}a.\] Writing $b_j=\rho(v_j)$ and $b_{j+1}=\rho(v_{j+1})$, we have $\rho(\phi^n(v_j))=\lambda^n b_j$ and $\rho(\phi^n(v_{j+1}))=\lambda^n b_{j+1}$.

Since these numbers are distance at most $\lambda^{2n}a$ apart, we have $\lambda^nb_j+(r+1)\lambda^na \geq \lambda^nb_{j+1}$. Hence if $s$ is the largest integer such that $\lambda^nb_j+s\lambda^n a<\lambda^nb_{j+1}$, then $s\leq r$ and between $\lambda^n b_j$ and $\lambda^n b_{j+1}$ there is a sequence of numbers \[\lambda^n b_j, \ \ \lambda^n b_j +\lambda^n a, \ \  \lambda^n b_j +2\lambda^n a, \ \ \ldots,  \ \ \lambda^n b_j +s\lambda^n a, \ \ \lambda^nb_{j+1}\] spaced at distances at most $\lambda^n a$ apart. Furthermore, for each $k$ between $0$ and $s$, we have \[\lambda^nb_j+k\lambda^n a=\rho(\phi^n(v_j))+k\rho(\phi^n(u))=\rho(\phi^n(v_j)+k\phi^n(u))=\rho(\phi^n(v_j+ku)).\] The element $v_j$ has word length at most $r$ with respect to $P$, and  $k\leq s\leq r$. Write $v_j$ as a word $g_1+\ldots+g_r$ where each $g_*$ lies in $P$ (with some possibly equal to 0). Then we have \[\phi^n(v_j+ku)=\phi^n(g_1+u)+\phi^n(g_2+u)+\cdots +\phi^n(g_k +u) +\phi^n(g_{k+1})+\cdots +\phi^n(g_r).\] Since $\phi^n(P+P)\subset P$, each term in this sum lies in $P$. Therefore $\phi^n(v_j+ku)\in rP$. Hence each number $\lambda^n b_j+k\lambda^n a$ where $k\leq s$ lies in $\rho(rP)$. This proves that $\rho(rP)$ intersects $[\lambda^nb_j,\lambda^nb_{j+1}]$ in a set of numbers spaced at distances at most $\lambda^na$ apart and containing both the endpoints. Since $0\leq j<t$ was arbitrary, this proves that $\rho(rP)$ intersects the entire interval $[\lambda^{in}a,\lambda^{(i+1)n}a]$ in a set of numbers spaced at distances at most $\lambda^na$ apart and containing both of the endpoints $\lambda^{in}a$ and $\lambda^{(i+1)n}a$. This completes the inductive step, and the proof of the claim.
\end{proof}
\end{proof}

\begin{lem}
\label{lem:unbounded}
Let $Q\subset \Z^2$ be confining under the action of $\phi$. Suppose that $\{b\in \R:av^++bv^-\in Q \textit{ for some } a\in \R\}=\pi(Q)$ is unbounded. Then $[Q\cup \{t^{\pm 1}\}]=[\Z^2 \cup \{t^{\pm 1}\}]$.
\end{lem}

\begin{proof}
Let $n$ be large enough that $\phi^n(Q+Q)\subset Q$; we fix this $n$ for the remainder of the proof. Denote $r=\lfloor \lambda^n \rfloor$. We claim that $\pi(rQ)$ is $1$--dense in $\R$, where $rQ$ denotes the set of words of length at most $r$ in the elements of $Q$. Before proving the claim, we show how it proves the lemma. By Lemmas \ref{lem:containnbhd} and \ref{equivsets}, there exists an upper bound $N$ on the word length of any element of $Q_1$ with respect to $Q$ (here $Q_1$ is the set $Q_\epsilon$ with $\epsilon=1$). Suppose that $g=cv^++dv^-\in rQ$. Then $g+Q_1=\{av^++bv^-:|b-d|\leq 1\}$. Since the set $\pi(rQ)$ is 1--dense in $\R$, this proves that every element of $\Z^2$ lies in $g+Q_1$ for some $g\in rQ$, and therefore every element of $\Z^2$ has word length at most $r+N$ with respect to $Q$.

Now we prove the claim. Let $g=cv^++dv^-\in Q$. Since $Q$ is symmetric, we may suppose without loss of generality that $d>0$. Let $k$ be the smallest integer with $\lambda^{-kn}d<1$. We claim that $\pi(rQ)$ is $(\lambda^{-kn}d)$--dense in the interval $[0,\lambda^{-kn/2}d]$. Since $\lambda^{-kn}d<1$, this will prove that $\pi(rQ)$ is in fact 1--dense in the interval $[0,\lambda^{-kn/2}d]$. Since $d$ may be taken to be arbitrarily large, the number $\lambda^{-kn/2}d$ may be taken to be arbitrarily large. This implies that $\pi(rQ)$ is 1--dense in $\R_{>0}$. Since $Q$ is symmetric, $\pi(rQ)$ will actually be 1--dense in all of $\R$, and this will complete the proof.

Thus we now show that $\pi(rQ)$ is $(\lambda^{-kn}d)$--dense in the interval $[0,\lambda^{-kn/2}d]$. The proof is similar to the proof of Claim \ref{denseprojs} and proceeds by induction. For the base case, note that all of the points  \[\lambda^{-kn}d=\pi(\phi^{kn}(g)),\ \ 2\lambda^{-kn}d, \ \ \ldots, \ \ r\lambda^{-kn}d\] lie in $\pi(r\phi^{kn}(Q))$, and they form a $(\lambda^{-kn}d)$--dense subset of $[0,\lambda^{-(k-1)n}d]$ since \[|\lambda^{-(k-1)n}d-r\lambda^{-kn}d|\leq |\lambda^{-(k-1)n}d-(\lambda^n-1)\lambda^{-kn}d|=\lambda^{-kn}d.\] In particular, these points are $(\lambda^{-kn}d)$--dense in $[\lambda^{-kn}d,\lambda^{-(k-1)n}d]$. 

For induction, suppose that for some $0\leq i\leq k/2-1$, we have a sequence of points \[b_0=\lambda^{-(k-i)n}d<b_1<b_2<\cdots<b_s\] in $[\lambda^{-(k-i)n}d,\lambda^{-(k-i-1)n}d]$ which all lie in $\pi(r\phi^{(k-2i)n}(Q))$ and are spaced at most $\lambda^{-kn}d$ apart. We wish to show that points of $\pi(r\phi^{(k-2i-2)n}(Q))$ are $(\lambda^{-kn}d)$--dense in $[\lambda^{-(k-i-1)n}d,\lambda^{-(k-i-2)n}d]$.

We have that \[\lambda^nb_0=\lambda^{-(k-i-1)n}d, \ \ \lambda^nb_1, \ \ \ldots, \ \ \lambda^nb_s\] all lie in $[\lambda^{-(k-i-1)n}d,\lambda^{-(k-i-2)n}d]$ and are points of $\pi(r\phi^{(k-2i-1)n}(Q))$. Namely, for each $j$, we may write $b_j=\pi(v_j)$ where $v_j \in r\phi^{(k-2i)n}(Q)$. We then have \[\lambda^nb_j=\pi(\phi^{-n}(v_j))\in\pi\left(\phi^{-n}(r\phi^{(k-2i)n}(Q))\right)=\pi\left( r\phi^{(k-2i-1)n}(Q)\right).\]

For each $j<s$, we have $|\lambda^n b_{j+1}-\lambda^n b_j|\leq \lambda^n \cdot \lambda^{-kn}d=\lambda^{-(k-1)n}d$. We consider the set of points \[A=\left\{\lambda^nb_j, \ \ \lambda^nb_j+\lambda^{-kn}d, \ \ \lambda^nb_j+2\lambda^{-kn}d, \ \  \ldots, \ \ \lambda^nb_j+r\lambda^{-kn}d\right\}.\] We claim that $A\subset \pi\left(r\phi^{(k-2i-2)n}(Q)\right)$. We have that \[\lambda^nb_j+(r+1)\lambda^{-kn}d\geq \lambda^n b_j +\lambda^{-(k-1)n}d \geq \lambda^nb_{j+1},\] so by this claim we  obtain a subset of $[\lambda^nb_j,\lambda^nb_{j+1}]$ consisting of points of $\pi(r\phi^{(k-2i-2)n}(Q))$ spaced at most $\lambda^{-kn}d$ apart, by considering the points $A\cap [\lambda^nb_j,\lambda^nb_{j+1}]$.

Write $v_j=h_1+\ldots+h_r$, where each $h_*\in \phi^{(k-2i)n}(Q)$ (some possibly equal to 0). Then for $l\leq r$ we have \[\lambda^nb_j+l\lambda^{-kn}d=\pi\left(\phi^{-n}(h_1)+\cdots+\phi^{-n}(h_r)+\underbrace{\phi^{kn}(g)+\cdots+\phi^{kn}(g)}_{l \text{ times}}\right).\] The expression inside $\pi$ is equal to \[\phi^{-2n}\left(\phi^n\left(h_1+\phi^{(k+1)n}(g)\right)+\cdots+\phi^n\left(h_l+\phi^{(k+1)n}(g)\right)+\phi^n(h_{l+1})+\cdots+\phi^n(h_r)\right).\] Since $\phi^n(Q+Q)\subset Q$, we also have $\phi^n\left(\phi^{(k-2i)n}(Q)+\phi^{(k-2i)n}(Q)\right)\subset \phi^{(k-2i)n}(Q),$ and therefore \[\phi^n\left(h_1+\phi^{(k+1)n}(g)\right),\ldots,\phi^n\left(h_l+\phi^{(k+1)n}(g)\right)\in \phi^{(k-2i)n}(Q).\] (Note that $\phi(Q)\subset Q$ implies that $\phi^{(k+1)n}(g)\in \phi^{(k-2i)n}(Q)$). Again using the fact that $\phi(Q)\subset Q$, we see that $\phi^n(h_{l+1}),\ldots, \phi^n(h_r)\in\phi^{(k-2i)n}(Q)$. Thus, finally, $\lambda^nb_j +l\lambda^{-kn}d\in \pi(r\phi^{(k-2i-2)n}(Q))$. 

Similarly, we may find points of $\pi(r\phi^{(k-2i-2)n}(Q))$ which form a $(\lambda^{-kn}d)$--dense subset of $[\lambda^nv_s,\lambda^{-(k-i-2)n}d]$. Thus, points of $\pi(r\phi^{(k-2i-2)n}(Q))$ form a $(\lambda^{-kn}d)$--dense subset of $[\lambda^{-(k-i-1)n}d,\lambda^{-(k-i-2)n}d]$. This completes the induction step  and the proof of the lemma.
\end{proof}

We are now ready to prove Theorem \ref{anosovposet}.
\begin{proof}[Proof of Theorem \ref{anosovposet}]
Since $G$ is solvable, it has no free subgroups, and therefore $\mc H_{gt}(G)=\emptyset$.  Thus it remains to describe the quasi-parabolic and lineal hyperbolic structures of $G$.

Note that the structure $[\Z^2 \cup \{t^{\pm 1}\}]$ is lineal, corresponding to the action of $G$ on $\R$ by translation. Namely $G$ admits a homomorphism $f:G\to \Z$ defined by taking the quotient by $\Z^2$, and the action on $\R$ is given by $g\cdot x= x+f(g)$.

Suppose that $[S]\in \H_{qp}(G)$. By Lemma \ref{lem:qpconfining} there is $Q\subset \Z^2$ which is confining under the action of $\phi$ or $\phi^{-1}$ such that $[S]=[Q\cup\{t^{\pm 1}\}]$. Suppose that $Q$ is confining under the action of $\phi$. Then  Lemma \ref{lem:containnbhd} implies that $[Q_\epsilon \cup \{t^{\pm 1}\}]\succcurlyeq [Q\cup \{t^{\pm 1}\}]$ (for any $\epsilon>0$). If $[Q\cup \{t^\pm 1\}]\neq[Q_\epsilon \cup \{t^{\pm 1}\}]$, then $[Q\cup \{t^{\pm 1}\}]=[\Z^2\cup \{t^{\pm 1}\}]$ by Lemma \ref{lem:unbounded}. However, this contradicts that $[Q\cup \{t^{\pm 1}\}]=[S]\in \H_{qp}(G)$. Hence $[Q\cup \{t^{\pm 1}\}]=[Q_\epsilon \cup \{t^{\pm 1}\}]$.

On the other hand, if $\H_{qp}(G)\ni [S]=[Q\cup \{t^{\pm 1}\}]$ where $Q$ is confining under the action of $\phi^{-1}$, then the same arguments show that $[Q\cup \{t^{\pm 1}\}]=[Q_\epsilon^-\cup \{t^{\pm 1}\}]$, where \[Q_\epsilon^-=\{av^++bv^- \in \Z^2 : |a|<\epsilon\}.\] (Note that all of the arguments for confining subsets under $\phi$ depended only on the fact that $\phi$ is an Anosov matrix; of course $\phi^{-1}$ is also an Anosov matrix).

The action $G\curvearrowright \Gamma(G,Q_\epsilon^- \cup \{t^{\pm 1}\})$ is equivalent to the following action of $G$ on $\hyp^2$: let $\Z^2$ act parabolically on $\hyp^2$ by $p\cdot z=z+\rho(p)$ and let $t$ act loxodromically on $\hyp^2$ by $t\cdot z=\lambda z$. By the argument in Lemma \ref{lem:hypplane}, this action of $G$ on $\hyp^2$ is equivalent to the action $G\curvearrowright \Gamma(G,Q_\epsilon^- \cup \{t^{\pm 1}\})$.

This can also be used to prove that $[Q_\epsilon \cup \{t^{\pm 1}\}]$ is not comparable to $[Q_\epsilon^- \cup \{t^{\pm 1}\}]$. Every conjugate of $t$ in the action of $G$ on $\Gamma(G,Q_\epsilon \cup \{t^{\pm 1}\})$ has a common \textit{repelling} fixed point, corresponding to $\infty$ in the upper half plane model of $\hyp^2$, but the conjugates of $t$ have many different attracting fixed points. On the other hand, every conjugate of $t$ in the action of $G$ on $\Gamma(G,Q_\epsilon^- \cup \{t^{\pm 1}\})$ has a common \textit{attracting} fixed point, but the conjugates have many different repelling fixed points. This suffices to show that the two actions are incomparable.

We have shown that $\mc H_{qp}(G)$ consists of two incomparable elements, $[Q_\epsilon \cup \{t^{\pm 1}\}]$ and $[Q_\epsilon^- \cup \{t^{\pm 1}\}]$, corresponding to two different actions of $G$ on $\hyp^2$.

Finally suppose that $[S]\in \H_\ell(G)$. We claim that $\Z^2$ acts elliptically on $\Gamma(G,S)$. Towards a contradiction, suppose otherwise.  Then the action of $\Z^2$ is cobounded, and \cite[Example~4.23]{abo} shows that the induced action of $\Z^2$ is lineal and fixes the two points of $\partial \Gamma(G,S)$.  Denote by $G_0$ the index $\leq 2$ subgroup of $G$ fixing the two points of $\partial \Gamma(G,S)$. Then $\Z^2\leq G_0$. If $t\in G_0$, then $G_0=G$. Otherwise, we have $t^2\in G_0$ and $G_0=\langle \Z^2, t^2\rangle$. Notice that, in this case, the group $G_0$ is isomorphic to $\Z^2\rtimes_{\phi^2} \Z$. In particular, in either case  Lemma \ref{abelianization} shows that the abelianization of $G_0$ is virtually cyclic, and therefore the commutator subgroup $[G_0,G_0]$ intersects $\Z^2$ in a finite-index subgroup. There is an associated \textit{Busemann homomorphism} $\beta\colon G_0\to \R$, and $\beta(g)\neq 0$ if and only if $g$ acts loxodromically on $\Gamma(G,S)$ (see \cite[Lemma~3.8]{ccmt}). This homomorphism factors through the abelianization of $G_0$, and in particular $\ker(\beta)\cap \Z^2$ has finite index in $\Z^2$ by the above discussion. But  since $\R$ is torsion free, it must be the case  that $\Z^2\leq \ker(\beta)$. Thus $\Z^2$ acts elliptically on $\Gamma(G,S)$. This is a contradiction to our assumption that the action of $\Z^2$ on $\Gamma(G,S)$ is \textit{not} elliptic. Thus, our claim that $\Z^2\curvearrowright \Gamma(G,S)$ is elliptic is proven.

We have shown that $[\Z^2 \cup \{t^{\pm 1}\}]\preccurlyeq [S]$. However, if two lineal structures are comparable then they must in fact be the same (\cite[Corollary~4.12]{abo}). Thus $[S]=[\Z^2\cup \{t^{\pm 1}\}]$ and $|\H_\ell(G)|=1$. Note that we have $[\Z^2\cup \{t^{\pm 1}\}]\preccurlyeq [Q_\epsilon \cup \{t^{\pm 1}\}]$ and $[\Z^2\cup \{t^{\pm 1}\}]\preccurlyeq [Q_\epsilon^-\cup \{t^{\pm 1}\}]$, which completes the proof.
\end{proof}

%%%%%%%%%%%%%%%%%%%%%%%%%%%%%%%%%%%%%%%%%%%%%%%%%%%%%%%%%%%%%%%%%%%

\bibliographystyle{plain}
\bibliography{inaccessible}

\noindent \textbf{Carolyn R. Abbott } Department of Mathematics, Brandeis University, Waltham, MA 02453. \\
E-mail: \emph{carolynabbott@brandeis.edu}

\bigskip

\noindent \textbf{Alexander J. Rasmussen } Department of Mathematics, University of Utah, Salt Lake City, UT 84112. \\
E-mail: \emph{rasmussen@math.utah.edu}

\end{document}